\documentclass[twoside]{article}
\usepackage{amsthm,amsmath,amsfonts,amssymb,bm,bbm,mathtools} 
\usepackage{graphicx}
\usepackage{color}
\usepackage[margin=30mm]{geometry}
\usepackage{datetime} \usdate
\usepackage{currfile} 

\newcommand{\ourtitle}
{Temporal oscillations in Becker-D\"oring equations\\
with atomization}
\author{Robert L. Pego$^{(1)}$,  Juan J. L. Vel\'{a}zquez$^{(2)}$ }

\newcommand{\myrunningheads}
{\pagestyle{myheadings}
\markboth
{Temporal oscillations in Becker-D\"oring equations} 
{R. L. Pego and J. J. L. Vel\'{a}zquez}
}
\myrunningheads

\newif\ifdraft
\draftfalse  

\ifdraft

\fi

\iffalse 
 \usepackage[landscape,includehead,headheight=12pt,headsep=10pt,left=\columnsep,right=\columnsep,top=12.98999pt,bottom=\columnsep,]{geometry}
 \twocolumn
\else
\fi

\newtheorem{theorem}{Theorem}[section]

\newtheorem{corollary}[theorem]{Corollary}
\newtheorem{lemma}[theorem]{Lemma}
\newtheorem{proposition}[theorem]{Proposition}

\theoremstyle{remark}
\newtheorem{remark}[theorem]{Remark}

\newcommand{\nwc}{\newcommand}

\nwc{\hide}[1]{}  
\nwc{\ip}[1]{\langle #1 \rangle}     
\nwc{\sfrac}[2]{{\textstyle \frac{#1}{#2}}}  
\nwc{\pdfx}[2]{\texorpdfstring{#1}{#2}}  

\nwc{\etal}{{\it et al.}}
\nwc{\D}{\partial}
\nwc{\eps}{\varepsilon}
\nwc{\inv}{^{-1}}

\nwc{\re}{\mathop{\rm Re}\nolimits}
\nwc{\im}{\mathop{\rm Im}\nolimits}
\nwc{\one}{{\mathbbm{1}}}
\nwc{\sgn}{\mathop{\rm sgn}\nolimits}

\nwc{\N}{\mathbb{N}}
\nwc{\Z}{\mathbb{Z}}
\nwc{\R}{\mathbb{R}}
\nwc{\CC}{\mathbb{C}}

\nwc{\bbD}{\mathbb{D}}
\nwc{\mb}[1]{\mathbf{#1}}
\nwc{\mc}[1]{{\mathcal #1}}

\nwc{\calA}{{\mathcal A}}
\nwc{\calB}{{\mathcal B}}
\nwc{\calD}{{\mathcal D}}
\nwc{\calF}{{\mathcal F}}
\nwc{\calV}{{\mathcal V}}
\nwc{\calI}{{\mathcal I}}
\nwc{\calZ}{{\mathcal Z}}

\nwc{\bmid}{\bm{\ell}}
\nwc{\vp}{\varphi}
\nwc{\am}{\gamma} 
\nwc{\ap}{\alpha} 

\nwc{\esssup}{\mathop{\rm ess\, sup}}

\newcommand{\kk}{\kappa} 
\newcommand{\calJ}{{\mathcal J}} 

\ifdraft
\title
{\ourtitle}
\date{WORKING DRAFT \today \\ PLEASE DO NOT CIRCULATE}

\else
\title
{\ourtitle}
\date{\today}
\fi

\usepackage{pdfsync,hyperref}
\begin{document}
\maketitle

\begin{center}
1-Department of Mathematics
and Center for Nonlinear Analysis\\
Carnegie Mellon University,
Pittsburgh, Pennsylvania, PA 12513, USA\\
email: rpego@cmu.edu
\end{center}

\begin{center}
2- Institut f\"ur Angewandte Mathematik\\
Universit\"at Bonn\\
Endenicher Allee 60\\
53115 Bonn, Germany\\
email: velazquez@iam.uni-bonn.de
\end{center}

\begin{abstract}  %
We prove that time-periodic solutions arise via Hopf bifurcation
in a finite closed system of coagulation-fragmentation equations.
The system we treat is a variant of the Becker-D\"oring equations,
in which clusters grow or shrink by addition or deletion of monomers.
To this is added a linear atomization reaction for clusters of maximum size.
The structure of the system is motivated by models of gas evolution oscillators
in physical chemistry, which exhibit temporal oscillations under certain
input/output conditions.
\end{abstract}

\medskip
\noindent
{\it Keywords. }{bubbling oscillator, shattering, bubbelator, time periodic solution}

\noindent
{\it Mathematics Subject Classification. } 34C23, 34A34, 82C99 

%
%

\section{Introduction}\label{s:intro}

Coagulation-fragmentation equations are commonly
used to model particle size distributions in
a wide range of scientific and technological applications. 
These equations model binary reactions
of clusters of size $j$ with clusters of size $k$
as indicated schematically by 
\begin{eqnarray*}
&&\hspace{-1cm}
(j) + (k) \, \stackrel{a_{j,k}}{\longrightarrow} \, (j+k) \quad \mbox{ (aggregation)}, \\
&&\hspace{-1cm}
(j) + (k) \, \stackrel{b_{j,k}}{\longleftarrow} \, (j+k) \quad \mbox{ (binary breakup)}, 
\end{eqnarray*}
With rate coefficients $a_{j,k}$ for aggregation and $b_{j,k}$ for breakup,
the net rate of this binary reaction is modeled by the law of mass action to be
\[
R_{j,k}= a_{j,k}\,n_jn_k - b_{j,k}\,n_{j+k}\,.
\]
The coagulation-fragmentation equations accounting for the 
gain and loss rates for the number density $n_j(t)$ of groups of size $j$
then take the form
\[
\D_t n_j = \frac12 \sum_{k=1}^{j-1} {R_{j-k,k}} 
- \sum_{k=1}^\infty {R_{j,k}} \,,\qquad j=1,2,\ldots
\]

To date, mathematical investigations of the dynamic behavior
of solutions have largely focused on questions of convergence to equilibrium
and the phenomenon of {\em gelation}, in which mass conservation fails 
(either in finite or infinite time) due to a flux to infinite size. 
We refer to classic work of Aizenman and Bak \cite{AB1979}
who established an $H$-theorem for perhaps the simplest coagulation-fragmentation model
with constant rate coefficients,
and Ball, Carr and Penrose \cite{BCP1986} for the first analysis
of (infinite-time) gelation in the Becker-D\"oring equations. 
If fragmentation is weak, finite-time gelation can occur 
\cite{ELMP,EMP,Laurencot2000,VZL1988}
as it does for the case of pure coagulation 
about which there is now an extensive literature.

Regarding convergence to equilibrium,
entropy methods have been effectively used to study
general classes of coagulation-fragmentation 
equations that admit equilibria in {\em detailed balance}, 
meaning that $R_{j,k}=0$ for each individual reaction in the system, 
so the forward and backward reaction rates match. 
See work of Lauren\c cot and Mischler \cite{LM2003} for the continuous-size case
and Ca\~nizo \cite{C-JSP2007} for the discrete-size case. 
More recent studies of equilibration have examined rates of convergence
and their relation to entropy-dissipation relations
\cite{JabinNiethammer,CanizoLods2013,MurrayPego2016,MurrayPego2017,CanizoLods2017}.

In the absence of detailed balance, however, one does not
expect that an $H$-theorem always holds, and it is not clear whether
the structure of coagulation-fragmentation reaction networks 
means that solutions necessarily always converge to some equilibrium.
Sometimes it is indeed the case, as in cases when coagulation is weak
\cite{FM2004} or for special systems that can be studied globally 
using transform methods, as in \cite{DLP2017}.
In \cite{LvR2015}, Lauren\c cot and van Roessel analyzed a model with a
critical balance of coagulation and fragmentation rates, and used transform
methods to show that infinite-time gelation emerges through self-similar
growth.  

On the other hand, in studies of pure coagulation without fragmentation,
the usual expectation of self-similar growth has sometimes been shown not to occur.  
For special rate kernels, solutions with fat tails are known 
to be capable of periodic and even chaotic behavior after rescaling \cite{MP2008}.
Temporal oscillations can persist after rescaling without fat tails
for Smoluchowski equations with diagonal rate kernel \cite{LNV2018}.

Our goal in the present work is to demonstrate that persistent 
oscillations in time are possible in a simple discrete-size
coagulation-fragmentation model, by proving that Hopf bifurcations occur. 

The particular system that we study is a 
modified system of Becker-D\"oring equations.
(For a nice historical review of mathematical developments concerning
the Becker-D\"oring equations, see \cite{HY2017}.)
As usual for Becker-D\"oring equations, 
we suppose that the coagulation of clusters of size $\ell$ with monomers
proceeds at the rate $a_\ell n_\ell n_1$, and clusters of size $\ell+1$ lose monomers
at the rate $b_{\ell+1} n_{\ell+1}$. 
We take these rates to apply only for a finite range of sizes $1\le \ell\le N$,
however, and consider only the simplest case, always taking $a_\ell = b_{\ell+1}=1$.
Thus the net flux of clusters from size $\ell$ to $\ell+1$ is
$J_\ell=R_{1,\ell}$, as given by 
\begin{equation}
J_{\ell}   = n_{\ell}n_{1} - n_{\ell+1} \,, \quad\mbox{for}\ 1\leq\ell\leq N ,
\label{d:Jell} 
\end{equation}
We suppose further that $M=N+1$ 
is the size of the largest clusters in the system,
and these are also subject to a {\em linear atomization} reaction that 
converts an $M$-cluster into $M$ monomers and proceeds at rate $Kn_M$.
Thus the governing equations take the following form:
\begin{align}
\D_t n_{\ell} &= J_{\ell-1} - J_{\ell} \,, \quad\mbox{for}\ 2\leq\ell\leq N,
\label{e:dtnell}\\[5pt]
\D_t n_{M} &= J_{M-1} - Kn_{M} \,, \quad M=N+1,
\label{e:dtnM}\\
\D_t n_{1}  &= -J_{1}-\sum_{\ell=1}^{N}J_{\ell}+MKn_{M} \,.
\label{e:dtn1}
\end{align}
All solutions of the system \eqref{e:dtnell}--\eqref{e:dtn1} conserve mass, since
\[
\partial_{t}\left(  \sum_{\ell=1}^{M}\ell n_{\ell}\right)  =0.
\]
A formal continuum analog of this system will be studied for illustrative purposes
in Section 4.

\section{Background and motivation}
{\em  Model with nonlinear atomization.}
In the physical literature, recent work of Matveev \etal~\cite{matveev2017}
and Brilliantov \etal~\cite{brilliantov2018}
has identified a coagulation-fragmentation model with a different,
{\em nonlinear} atomization mechanism that exhibits persistent temporal oscillations
in numerical simulations.  In this model, aggregation of clusters of size $i$ and $j$ 
proceeds at rate $a_{i,j}n_in_j$ where
\[
a_{i,j} = (i/j)^\alpha + (j/i)^\alpha
\]
and pairs of such clusters atomize upon collision into 
$i+j$ monomers at rate $\lambda a_{i,j}n_i n_j$. In total, the rate equations
in \cite{matveev2017} take the form
\begin{align}
\D_t n_1 &= -\sum_{i=1}^\infty a_{1,i}n_1n_i 
+ \lambda \sum_{j=2}^\infty j a_{1,j}n_1n_j
+ \frac\lambda2 \sum_{i=2}^\infty \sum_{j=2}^\infty (i+j)a_{i,j}n_i n_j,
\\
\D_t n_k &= \frac12 \sum_{i=1}^{k-1} a_{i,k-i} n_i n_{k-i} - (1+\lambda)\sum_{i=1}^\infty a_{i,k}n_in_k,
\quad k\ge 2.
\end{align}
This system has the feature that 
interactions between large clusters of similar size appear to be dominated
by interactions between large clusters and small ones 
(for which either $i/j$ or $j/i$ is large).
Oscillations are found for $\frac12< \alpha \leq1$ and small $\lambda>0$.
Though the numerics is convincing, to our knowledge there is no proof yet that 
temporal oscillations persist in this system.

{\em Bubbling oscillators.}
Our motivation for studying the system \eqref{e:dtnell}--\eqref{e:dtn1} comes from 
literature in physical chemistry concerning {\em bubbling oscillators}
(often called `gas evolution oscillators' in much of the literature). 
In these systems, dissolved gas (such as CO or CO$_2$) is {\em added} slowly
to a liquid solution, producing a super-saturated mixture.  At some time,
nucleation of gas bubbles occurs spontaneously and the bubbles grow
rapidly and carry most of the dissolved gas {\em out of the system}. 
The first system of this kind was reported by J.~S.~Morgan in 1916, who found
that a small concentration of formic acid mixed in sulphuric acid 
produced periodic bursts of carbon monoxide. 
Such systems were the subject of part of an extensive series of  quantitative studies
by R.~M.~Noyes and collaborators concerning chemical oscillators,
including some of the original studies of 
chemical oscillators such as the BZ reaction and the Oregonator. 
Regarding gas evolution oscillators, we especially refer to \cite{smith1983,YRN1985,bar1992}.
The phenomenon of sudden outgassing of CO$_2$ after slow buildup of supersaturation
was responsible for the 1986 Lake Nyos disaster in Cameroon, which killed more than 1700 people.

In the work of Yuan, Ruoff and Noyes \cite{YRN1985}, this process was simulated
numerically by grouping bubble sizes into a finite set corresponding to 
exponentially spaced radii $r_j$, and writing rate equations to model the 
number density $N_j$ of bubbles of size $r_j$. A key equation when $r_j$
 greater than a critical value $r_{\rm eq}$ is
\begin{equation}
\D_t N_j = q_{j-1}N_{j-1} - (q_j+k_j)N_j \,,
\end{equation}
where the coefficients $q_j$ are proportional to bubble growth rate
and the $k_j$ are rate constants for escape.
This resembles a linearized Becker-D\"oring equation or a discretized advection equation,
and models the process of free bubble growth and escape. With $M=60$ size classes,
numerical simulations in \cite{YRN1985} exhibit temporal oscillations for 
a range of parameters designed to model experimental conditions. 

Bar-Eli and Noyes \cite{bar1992} later devised a simplified,
qualitative model for bubbling oscillators that involves a nonlinear {\em differential-delay equation}
for the concentration of dissolved gas.  When linearized about a constant steady-state,
one obtains a constant-coefficient linear DDE of the form
\begin{equation}
\partial_t x(t) = -a x(t-\tau) - b x(t) \,,
\label{dde}
\end{equation}
where the parameters $a$, $b$ and the delay time $\tau$ are positive constants. 
Whenever $a>b$, one finds there is an oscillatory transition from stability to 
instability as $\tau$ increases.

We sketch loosely how one can see this mathematically. 
(A detailed analysis of \eqref{dde} can be found in work of
Hadeler and Tomiuk \cite{HadelerTomiuk}.)
Equation \eqref{dde} has solution $e^{\lambda t}$ provided 
\begin{equation}
\rho(\lambda):= a e^{-\lambda \tau} + b + \lambda = 0. 
\label{rho0}
\end{equation}
For $\tau=0$, naturally $\lambda<0$, and moreover $\lambda=0$ is never possible for any $\tau$. 
But for $a>b$ and $\tau$ sufficiently large, 
there are solutions with $\re\lambda>0$.  To show this is so, 
one can consider the winding number around 0 of a curve $\rho\circ \gamma$,
where $\gamma$ is a concatenation of a path $s\mapsto -is$ for $s\in[-R,R]$
and a path in the right half plane along the semicircle where $|\gamma|=R>a+b$. 
Along the semicircle, $\rho\circ\gamma$ can never cross the negative real axis
$\R_-$. Along the imaginary axis, however,
\[
\rho\circ\gamma(s) = a e^{is\tau} + b - is,
\]
and this does cross $\R_-$ for $s$ between $0$ and $2\pi/\tau$,
if $\tau$ is large enough. Moreover, $\rho\circ\gamma(s)$ can only 
ever cross $\R_-$ going from the second quadrant to the third, since whenever
$\rho\circ\gamma(s)<0$, 
\[
\frac{d}{ds} \rho\circ\gamma(s) = i\tau(\rho\circ\gamma(s)-b+is) - i 
\]
and this has negative imaginary part. Consequently,
the winding number of $\rho\circ\gamma$ around 0 is positive if $\tau$ is 
large enough, and this implies \eqref{rho0} has a root $\lambda$ inside $\gamma$.

{\em Becker-D\"oring with linear atomization (our model).}
Now, the rough idea behind our model \eqref{e:dtnell}--\eqref{e:dtn1}
is that the Becker-D\"oring equations involve a well-known advection mechanism
that transports mass from small cluster sizes to large ones when the monomer
concentration is supercritical. 
The atomization reaction added in \eqref{e:dtnM} couples the advected wave back to the monomer
concentration after a time delay that depends on the size of the system. 
Luckily enough, we find that for large $M$ there indeed is an oscillatory transition to 
instability as the parameter $K$ varies, in a certain parameter range
where $K$ is small but $KM$ remains large. 
See Figure~\ref{f:n1vt},
where we plot the monomer concentration {\it vs.}\ time
for a numerically computed solution of \eqref{e:dtnell}--\eqref{e:dtn1}
with parameters and initial values given by 
\begin{equation}\label{fig1pars}
M=25,\quad K=3, \quad n_1=4.2,\quad n_\ell=1+K=4 \quad\mbox{ for $\ell\geq2$.}
\end{equation}
\begin{figure} \begin{center}
\includegraphics[width=4.7in]{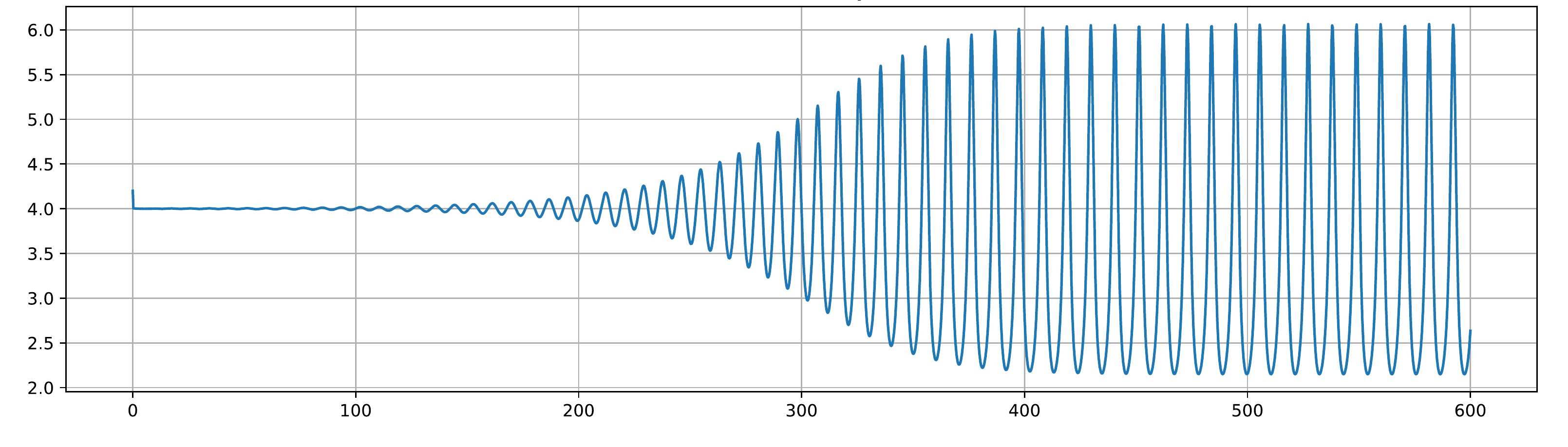}
\end{center}
\caption{Monomer density $n_1$ vs. $t$ for a numerical solution of 
\eqref{e:dtnell}--\eqref{e:dtn1}.}
\label{f:n1vt}
\end{figure}

{\em Other models with linear atomization.}
Finally, we mention two other kinds of merging-splitting models 
involving a linear atomization reaction that have appeared in the literature. 
Alongside discussion of Niwa's model \cite{Niwa-JTB2003} for animal group size,
Ma \etal~\cite{Ma_etal_JTB11} described a ``preferential attachment'' model, which takes the form
\begin{equation}
\D_t n_j = (j-1)n_1n_{j-1} - j n_1n_j - Kj n_j, \quad\mbox{for}\ j\geq 2.
\end{equation}
This model admits a simple logarithmic distribution in equilibrium, of the form
\[
n_j = \frac{e^{-\beta j}}j, \qquad e^{-\beta} = \frac{n_1}{n_1+K}.
\]
(This is roughly similar to the distribution Niwa found to be a good description
of empirical data on school size for pelagic fish.) 
A model of herd behavior by networks of colluding agents in financial markets
introduced by Eguiluz \& Zimmermann \cite{eguiluz2000} takes the form
\begin{equation}
\D_t n_j = \sum_{k=1}^{j-1} (j-k)k n_{j-k}n_k - 2 \sum_{k=1}^\infty jk n_j n_k - Kjn_j,
\quad j\geq 2.
\end{equation}
D'Hulst and Rodgers \cite{d2000exact} found a formula for equilibrium solutions of this model
by use of generating functions. But as far as we are aware, 
no analysis of dynamics has been carried out for either of these models. 

\section{Equilibria, linearization, and main result}

In this section, we find the general equilibrium solutions of the model
\eqref{e:dtnell}--\eqref{e:dtn1}, describe the special family of constant equilibria,
and state our main rigorous result on the existence of Hopf bifurcations from
this family, occurring at particular values of $K$, for large enough $M$.

\subsection{General equilibria}
We find the general equilibria as follows. 
In equilibrium, due to \eqref{e:dtnell} the fluxes $J_\ell$ are all equal
to the same value $J$ for $\ell=1,\ldots,N$, 
so the equilibrium number densities $\bar n_\ell$ satisfy the difference equation
$\bar n_{\ell+1} = z\bar n_\ell-J$, when we require $\bar n_1=z$. 
For $z\neq1$, the solution takes the form
\begin{equation}
\bar n_\ell =  z(1-\alpha) + z^\ell\alpha,  
\quad\mbox{where}\quad 
J= (z^2-z)(1-\alpha).
\end{equation}
To obtain an equilibrium it remains to require that \eqref{e:dtnM} hold, i.e.,
\[
0 = J-K\bar n_M = (z^2-z-Kz)(1-\alpha) - Kz^M\alpha .
\]
Then it follows (recall $N=M-1$)
\begin{equation}
\alpha = \frac{z-1-K}{Kz^{N}+z-1-K}\,,
\label{e:alpha}
\end{equation}
and 
\begin{equation}
\bar n_\ell = \frac{Kz^M+z^\ell(z-1-K)}{Kz^N+z-1-K}, \quad \ell=1,\ldots,M.
\label{e:neq}
\end{equation}
Note that \eqref{e:dtn1} then holds also. For every $z>0$, $N>1$ and $K>0$,
such an equilibrium exists and is positive.
In case $z=1$, one finds directly that 
\begin{equation}
\bar n_\ell = \frac{1+(M-\ell)K}{1+NK}, \quad \ell=1,\ldots,M.
\label{d:nell}
\end{equation}

The total mass as a function of $z$ and $K$ is now
\[
m = \sum_{\ell=1}^M \ell n_\ell = 
\alpha \mu_M(z) + (1-\alpha)z\mu_M(1)
\]
where 
\[
\mu_M(z) = \sum_{\ell=1}^M \ell z^\ell = z \frac{d}{dz}\frac{1-z^{M+1}}{1-z}
= \frac{z-z^{M+1}}{(1-z)^2} -\frac{Mz^{M+1}}{1-z}.
\]

\subsection{Linearization at constant equilibria}
Particularly convenient for our analysis is the special family
of equilibria that have constant densities, corresponding to $\alpha=0$.
By \eqref{e:alpha} these take the form
\begin{equation}\label{d:barn}
\bar n_{\ell}=A:=1+K\,,\quad 1\leq\ell\leq M.
\end{equation}
Corresponding to $K>0$ we require $A>1$.
We will study the linearization of the system \eqref{e:dtnell}--\eqref{e:dtn1} about
this equilibrium. We write:%
\[
n_{k}=A+v_{k}\,, \quad 1\leq\ell\leq M.
\]
The linearized fluxes take the form%
\[
L_{\ell}=Av_{1}+Av_{\ell}-v_{\ell+1}\,, \quad 1\leq\ell\leq N=M-1,
\]
and the linearized evolution equations are written as follows:%
\begin{align}
\partial_{t}v_{\ell}  &  =L_{\ell-1}-L_{\ell}\,, \quad 2\leq\ell\leq N,
\label{L:2N}\\[5pt]
\partial_{t}v_{M}  &  =L_{M-1}-Kv_{M}\,, 
\label{L:M}\\
\partial_{t}v_{1}  &  =-L_{1}-\sum_{\ell=1}^{N}L_{\ell}+MKv_{M} .
\label{L:1}
\end{align}
Equivalently, after some computations, the system takes the more explicit form%
\begin{align}
\partial_{t}v_{\ell} &= K\left(  v_{\ell-1}-v_{\ell}\right)  +\left(  v_{\ell
-1}-2v_{\ell}+v_{\ell+1}\right) \,,\quad 2\leq\ell\leq N,
\label{e:vell}\\[5pt]
\partial_{t}v_{M} &= \left(  K+1\right)  \left(  v_{1}+v_{N}-v_{M}\right) ,
\label{e:vM}\\
\partial_{t}v_{1} &= 
-A\left(N+3\right)  v_{1}+v_{2}-K\sum_{\ell=2}^{N}v_{\ell}+\left(  MK+1\right)  v_{M}.%
\end{align}
Equation \eqref{e:vell} yields a combination of diffusion and transport. It is not able
to yield oscillatory behavior by itself, but this will be generated through the
`boundary conditions,' or more precisely the equations with $\ell=M$ and $\ell=1.$

Looking for solutions of this system with the form%
\[
v_{\ell}=V_{\ell}e^{\lambda t}\,, \quad V_{\ell}\in\mathbb{C},%
\]
leads to the eigenvalue problem (recalling $A=1+K$ and $M=N+1$)
\begin{align}
\lambda V_{1}  &= -A\left(  N+3\right)  V_{1} + V_{2} 
    - K\sum_{\ell=2}^{N} V_{\ell}+\left(  MK+1\right)  V_{M} \,,
\label{v:v1}\\
\lambda V_{\ell}  &= K\left(  V_{\ell-1}-V_{\ell}\right)  
    + \left(  V_{\ell-1}-2V_{\ell}+V_{\ell+1}\right)\,,  \quad 2\leq\ell\leq N\,,
\label{v:vell} \\[6pt]
\lambda V_{M}  &= \left(  K+1\right)  \left(  V_{1}+V_{N}-V_{M}\right)\,.
\label{v:vM}
\end{align}

This system takes the form of an eigenvalue problem $BV=\lambda V$
for a vector $V=(V_1,\ldots,V_M)^T \in \CC^M$, with $M\times M$ matrix
$B$ having the structure
\begin{equation}
B  = 
\begin{pmatrix}
-A(M+2) & 1-K & -K && \cdots & -K & MK+1 \\
A & -A-1 & 1 & 0 &\cdots & 0 & 0 \\
0 & A & -A-1 & 1 & & & \\
 & 0 & A & & & & \\
\vdots &\vdots && \ddots &\ddots&\ddots \\
0 & 0 & & \cdots & A & -A-1 & 1 \\
A &0 & & \cdots & 0 & A & -A
\end{pmatrix}
\end{equation}
Our goal is to understand the spectrum of $B$ in some detail, and eventually
show that in a certain parameter range, some pair of complex eigenvalues
of $B$ crosses the imaginary axis, and this forces the system
\eqref{e:dtnell}--\eqref{e:dtn1} to undergo a Hopf bifurcation.

To begin, 
one can check  that for any eigenvector
$V=(V_1,\ldots,V_M)^T$ corresponding to a nonzero eigenvalue
$\lambda\neq0$, the mass conservation condition holds:
\begin{equation}
\sum_{\ell=1}^M \ell V_\ell = 0.
\label{v:mass0}
\end{equation}
This is due to the fact that $\bmid= (1,2,\ldots,M)$ is a left null vector of $B$. 
Thus $\lambda=0$ is an eigenvalue of $B$.  If we represent
the equations of equilibrium from \eqref{e:dtnell}--\eqref{e:dtn1} 
in vector form as $\calF(\bar n)=0$, then the matrix $B=\D\calF/\D n$ 
evaluated at $\bar n$ with constant components $A$. Thus by differentiation it naturally 
follows from \eqref{e:neq} 
that a right null vector $\bar v$ satisfying $B\bar v=0$ is given by 
\begin{equation}
\bar v_\ell = 
 \left. \frac{\D \bar n_\ell}{\D z}\right|_{z=A} =
M + \frac{A^\ell}{KA^N} - A \frac{KNA^{N-1}+1}{KA^N} = 1+\frac{A^\ell-A}{KA^N}
\,, \quad 1\leq \ell\leq M.
\label{v:null}
\end{equation}

In fact, we have the following.
\begin{lemma}\label{lem:simple0}
For all $N>1$ and $K>0$, 
$\lambda=0$ is a {simple} eigenvalue of $B$.
\end{lemma}
\begin{proof}
First, we show the null space of $B$ is one-dimensional.
Whenever $BV=0$, the fluxes defined by 
\begin{equation}
L_\ell = AV_1+AV_\ell-V_{\ell+1}\,, \quad 1\leq \ell\leq N, 
\label{d:LV}
\end{equation}
must all take the same value due to \eqref{L:2N},  
and for $V=\bar v$ this value is $K\bar v_M>0$ due to \eqref{L:M}.  
If $BV=0$, then we can replace $V$ by a linear combination with $\bar v$
to make all fluxes $L_\ell=0$. 
But by \eqref{d:LV} it follows $V_2=2A V_1$ and, by induction,
$V_\ell=A_\ell V_1$ with $A_\ell>0$ for $\ell=2,\ldots,M$.
Since $0=L_N=KV_M$,
the only vector $V$ making all the fluxes vanish is $V=0$. 
It follows that $\bar v$ spans the null space of $B$.

Next, we claim there is no generalized eigenvector $V$ satisfying 
$BV=\bar v$.  The reason is that, because $\bar v_\ell \geq 1>0$ for all $\ell$
and $\bmid$ is a left null vector of $B$, we would obtain a contradiction, via
\[
0<
\sum_{\ell=1}^M \ell \bar v_\ell 
= \bmid \bar v = \bmid(B V) = (\bmid B)V=0.
\]
Thus the eigenvalue $\lambda=0$ has algebraic multiplicity one, so it is simple.
\end{proof}

\subsection{Main results}

If $\lambda$ is an eigenvalue of $B$, we say $\lambda$ is {\em unstable} if $\re\lambda>0$.
We find that we can show the matrix $B=B(K,M)$ has unstable eigenvalues
when $M$ is sufficiently large and $K$ is small but not too small, 
in a range proportional to $1/\sqrt{M}$.
These eigenvalues are characterized as follows. 
It is convenient to state our results in terms of the parameter
\begin{equation}\label{d:kappa}
\kappa = K\sqrt{M} \,,
\end{equation}
in place of $K=\kappa/\sqrt{M}$.

\begin{theorem} \label{thm:eigsB}
For each $k\in\N$ and $\beta_0\in(0,1)$, there exists 
$\beta_k>\beta_0$, and positive constants $M_{0,k}$, $\hat C_k$, 
such that for each $M>M_{0,k}$, the following hold:
\begin{enumerate}
\item If $\beta_0<\kk<\beta_k$, then any unstable eigenvalue of $B$ is non-real and simple, and satisfies
\[
|\lambda|\le \hat C_k M^{-3/2}.
\]
\item There are numbers $\kappa_j=\kappa_j(M)$ for $j=1,\ldots,k$,
satisfying 
\[
\beta_0 =:\kappa_0 <\kappa_1<\ldots<\kappa_k<\kappa_{k+1}:=\beta_k \,,
\]
such that:
\begin{enumerate}
\item 
If $\kappa_j<\kappa<\kappa_{j+1}$ ($j=0,\ldots,k$),
then $B$ has exactly $j$ complex-conjugate
pairs $(\lambda,\bar\lambda)$ of unstable eigenvalues. 
\item There are analytic curves $\lambda_j:[\beta_0,\beta_k]\to\CC$, $j=1,\ldots,k$, such that
$\lambda_j(\kappa)$ is an eigenvalue of $B$ that satisfies 
\begin{equation}
\re\lambda_j(\kappa_j)=0,  \qquad
\im \lambda_j(\kappa)>0
\quad\mbox{for all $\kappa\in [\beta_0,\beta_k]$,}
\end{equation}
along with 
\begin{equation}\label{e:dlamdk}
\re \frac{d\lambda_j}{d\kappa}>0 \,, \quad
\im \frac{d\lambda_j}{d\kappa}>0 \,, \quad
\quad\mbox{for all $\kappa\in[\kappa_j,\beta_k]$.}
\end{equation}
\end{enumerate}
\end{enumerate}
\end{theorem}

By the properties stated in part 2 of this theorem, the matrix
$B$ has a unique pair of nonzero, purely imaginary eigenvalues 
$\pm\lambda_j(\kappa_j)$ when $M$ is large and $\kappa=\kappa_j(M)$,
and these cross transversely into the right half plane as $K$ increases.

The simple eigenvalue at zero, described in Lemma~\ref{lem:simple0}, 
is nominally an obstruction to applying the standard Hopf bifurcation theorem
at this point. 
This eigenvalue is easily removed, however, by considering the dynamics 
of the nonlinear system \eqref{e:dtnell}--\eqref{e:dtn1} 
restricted to the invariant affine hyperplane determined by conservation of mass,
i.e., the hyperplane where 
\begin{equation}
\sum_{\ell=1}^n \ell n_\ell = 
\sum_{\ell=1}^n \ell \bar n_\ell  
\end{equation}
with $\bar n=(\bar n_\ell)$ 
being the constant equilibrium state from \eqref{d:barn}.
Within this hyperplane, the linearization of the 
system \eqref{e:dtnell}--\eqref{e:dtn1} 
is restricted to orthogonal complement of the left null vector $\bmid$ of $B$.
In this subspace, the zero eigenvalue is removed, and the 
standard Hopf bifurcation theorem can be applied to yield the following result.
(For a proof of the Hopf bifurcation theorem see \cite[pp.~98--99]{ChowHale}.
For further discussion also see \cite[pp.~150--152]{GH}.)

\begin{theorem}\label{thm:hopf}
Let $k\in\N$ and suppose $M>M_{0,k}$ as given by Theorem~\ref{thm:eigsB}.
Then for each $j=1,\ldots,k$, the system \eqref{e:dtnell}--\eqref{e:dtn1} admits a Hopf bifurcation 
as the bifurcation parameter $\kappa$ passes through $\kappa_j=\kappa_j(M)$.
Thus a time-periodic solution exists
for some value of $\kappa$ with $|\kappa-\kappa_j|$ small. 
\end{theorem}

We have not managed to determine analytically whether the bifurcating 
solutions are stable (the supercritical case) or not. Many of our
numerical computations, as in Fig.~\ref{f:n1vt},
are consistent with the presence of stable periodic solutions, however.

\begin{figure} \begin{center}
\includegraphics[width=4.5in]{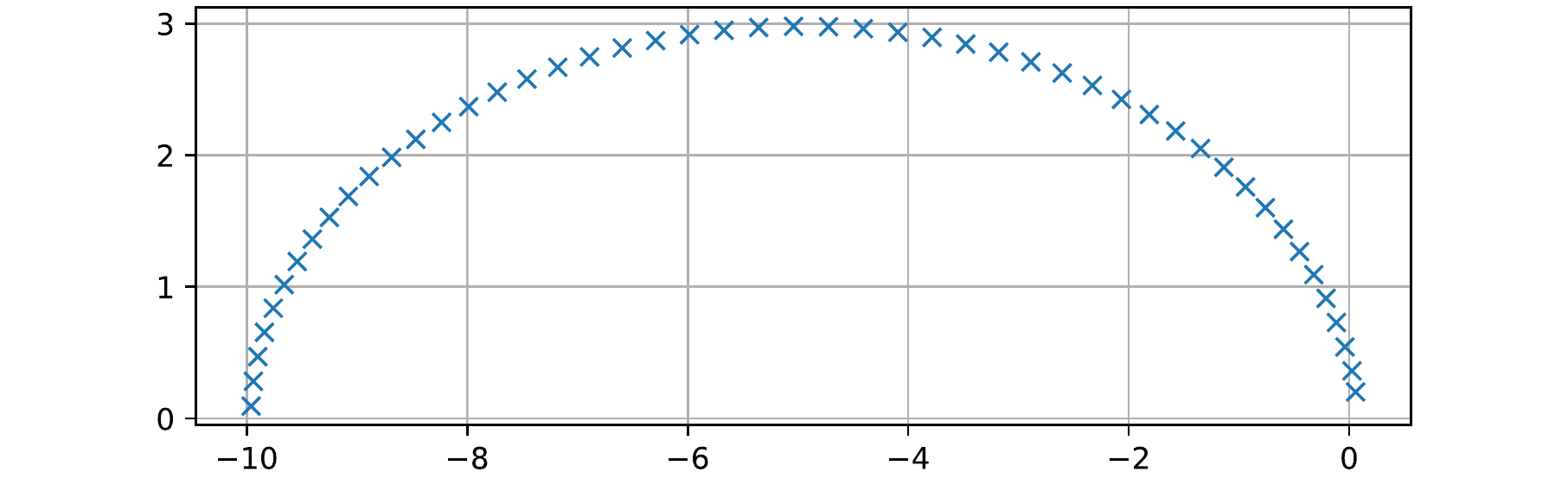}
\end{center}
\caption{Complex eigenvalues of matrix $B$ for $M=100$, $K=3$}
\label{f:evalsM100K3}
\end{figure}

In Figure~\ref{f:evalsM100K3} we illustrate the location of 
the complex eigenvalues of $B$ computed numerically for the parameter
values $M=100$ and $K=3$. The unstable eigenvalues shown correspond
to values
\[
\lambda \approx 0.05836 \pm 0.2014i, \quad
0.02585 \pm 0.3618i.
\]
Besides the real eigenvalue $\lambda=1$ 
this matrix also has a large negative eigenvalue $\lambda\approx-410.94$. 
There are 49 complex-conjugate pairs of eigenvalues that lie close
to an ellipse that we will describe formally in Remark~\ref{r:ellipse}
and Section~\ref{s:formal}. 
The presence of eigenvalues spaced closely along a smooth curve
and a large isolated eigenvalue is reminiscent of the structure of
the spectrum of differential-delay systems with large delay, see
\cite{Lichtner2011}.

The real parts of eigenvectors for the first 3 complex eigenvalues
closest to $\lambda=1$ are plotted in Figure~\ref{f:evecs}.
They appear to have a ``smooth'' structure 
except in a boundary layer near $\ell=M$.
\begin{figure}  \begin{center}
\includegraphics[width=4.5in]{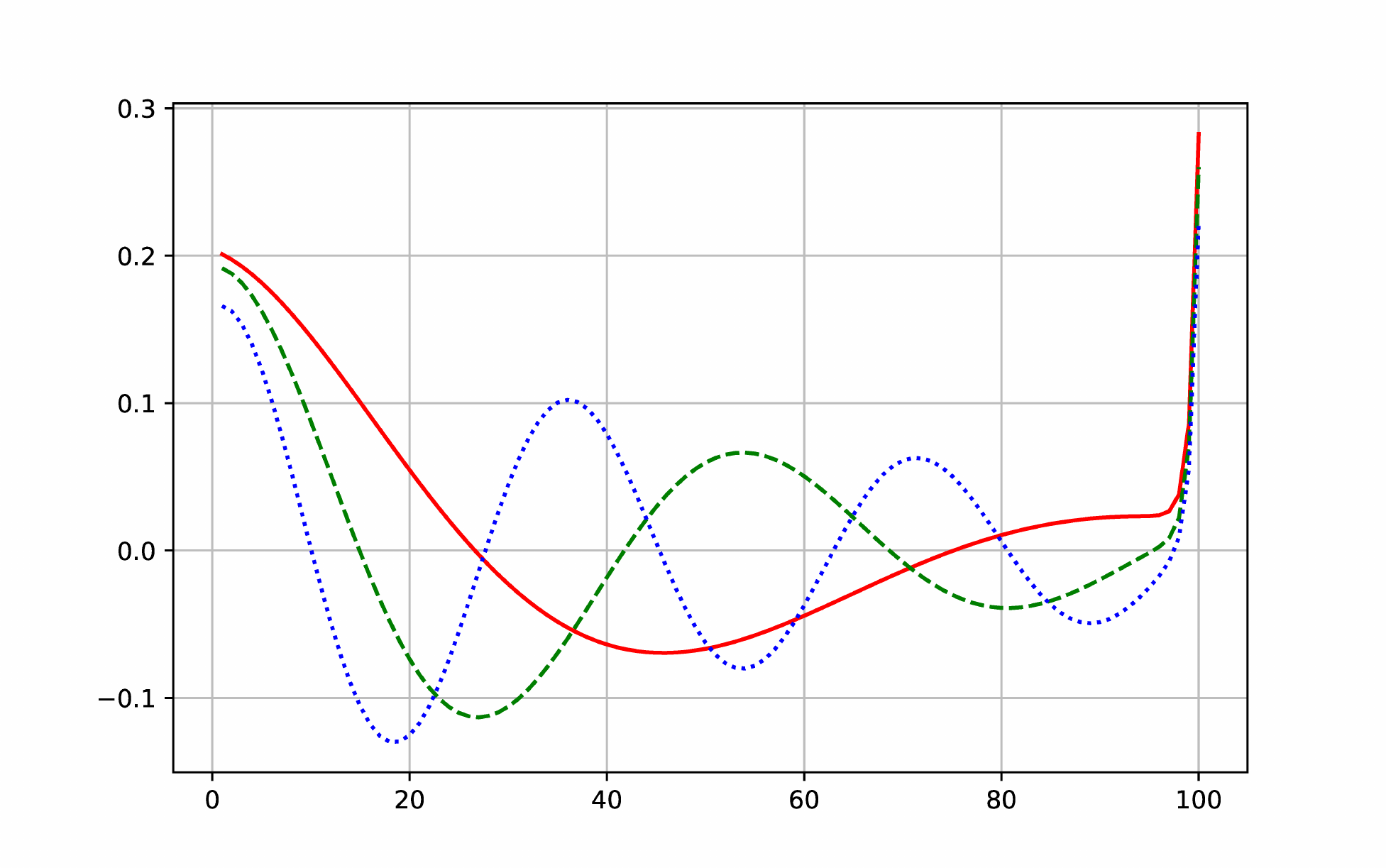}
\end{center}
\caption{Eigenvectors for 3 eigenvalues near $\lambda=1$ for $M=100$, $K=3$}
\label{f:evecs}
\end{figure}

\begin{table}[ht]
\caption{Critical parameters for first Hopf bifurcation}
\begin{center}
\begin{tabular}{cccc}
    \hline
    $M$ & $K$ & $\kappa_1=K\sqrt{M}$ & $\im\lambda$ \\
    \hline
    $10^2$ & 0.39349 & 3.9349   &0.021740\\
    $10^3$ & 0.075016 & 2.3722  & 3.6176e-4 \\
    $10^4$ & 0.020376 & 2.0376  & 9.3596e-6\\
    $10^5$ & 0.0061392 & 1.9414 & 2.7777e-7\\
    $10^6$ & 0.0019118 & 1.9118 & 8.6091e-9\\
\end{tabular}
\end{center}
\label{tab:multicol}
\end{table}

In Table 1 we tabulate for various values of $M$ 
numerically computed critical values of $K$ that correspond to $\kappa_1$,
the value at which the first pair of complex-conjugate eigenvalues crosses
the imaginary axis. 
Eigenvalues were obtained by solving the equation \eqref{e:Feq0} in
Proposition~\ref{p:lambdaphi} below using an iteration method.
The first two rows were computed also
by finding all eigenvalues of $B$ using the julia function eigen.
The values of $\kappa_1$ in the third column can be compared to the 
value $\kappa_{\rm cr}\approx 1.89825$ described below in 
\eqref{e:kappacr}. This value is proved in Section~\ref{s:eigsB}
to be the limiting value of $\kappa_1$ as $M\to\infty$, 
see \eqref{e:kappalim}.

\section{A formal continuum approximation}

Here we describe and heuristically analyze a formal continuum approximation
of the system \eqref{e:dtnell}--\eqref{e:dtn1} that we are studying.
This will serve to preview how the proof of Theorem~\ref{thm:eigsB} will proceed,
and also provides an approximate understanding of 
the origin of the oscillatory instability in the system 
and the smooth structure of the
eigenvectors as shown in Figure~\ref{f:evecs}. 

We introduce a scaled atomization rate and scaled space and time variables via
\begin{equation}\label{d:xtau}
K = \eps\kappa,\quad x = \eps^2 \ell, \quad \tau = \eps^3 t, 
\qquad \eps=\frac1{\sqrt{M}}.
\end{equation}
With these relations, we 
write a continuum approximation of \eqref{e:dtnell}--\eqref{e:dtn1} 
in terms of variables
\[
u(x,\tau)\approx n_\ell(t), \qquad \calJ(x,\tau) \approx J_\ell(t),
\]
as follows: 
The evolution equations for $\D_t n_\ell$, $\ell=2,\ldots,M$ 
in \eqref{e:dtnell}--\eqref{e:dtnM} are approximated by the PDE
\begin{equation}\label{e:Dtu}
\eps^3\D_\tau u +\eps^2\D_x\calJ = 0, \qquad 0<x<1, 
\end{equation}
where the number flux relation $J_\ell = (n_1-1)n_\ell +n_\ell-n_{\ell+1}$ from \eqref{d:Jell} is approximated by 
\begin{equation}\label{d:calJell}
\calJ(x,\tau) = (u(0,\tau)-1)u(x,\tau) - \eps^2\D_x u(x,\tau).
\end{equation}
Taking \eqref{e:dtnell} to hold also for $\ell=M$ with $J_M=Kn_M$,
the following boundary condition replaces \eqref{e:dtnM}:
\begin{equation}
\calJ(1,\tau) = K u(1,\tau)
\end{equation}
We replace the evolution equation for $\D_t n_1$ by an equation 
equivalent to mass conservation, which says
\[
0 = \eps\frac{d}{d\tau}\int_0^1 xu(x,\tau)\,dx = 
-\int_0^1 x\D_x\calJ(x,\tau)\,dx.
\]
After integrating by parts, we require 
\begin{equation}\label{e:Mxcons}
\calJ(1,\tau) = \int_0^1\calJ(x,\tau)\,dx.
\end{equation}

The system of equations \eqref{e:Dtu}--\eqref{e:Mxcons} has
the constant equilibrium $u(x)=A=1+K$ just like the discrete system. 
Using a superposed dot to denote differentiation with respect to a 
variational parameter leads to the linearized system
\begin{align}
0 &= \eps \D_\tau \dot u + \D_x\dot\calJ,
\\
 \dot\calJ(x,\tau) &= (1+\eps\kk)\dot u(0,\tau) + \eps\kk\dot u(x,\tau)-\eps^2\D_x\dot u(x,\tau),
\\ \dot\calJ(1,\tau) &= \eps\kk\dot u(1,\tau),
\\ \dot\calJ(1,\tau) &= 
(1+\eps\kk+\eps^2) \dot u(0,\tau) +\eps\kk\int_0^1\dot u(x,\tau)\,dx - \eps^2\dot u(1,\tau).
\end{align}
Then looking for solutions of the form $e^{\Lambda \tau}v(x)$ 
leads to the eigenvalue problem
\begin{align}
0  &= \Lambda v + \D_x(\kappa v -\eps\D_x v), \qquad 0<x<1,
\label{e:vx}
\\0&= (1+\eps\kk)v(0)-\eps^2\D_x v(1),
\label{e:v0}
\\ 0&= (\eps\kappa+\eps^2)v(1)-\eps\kappa\int_0^1v(x)\,dx -\eps^2\D_x v(1)\frac{1+\eps\kappa+\eps^2}{1+\eps\kappa}.
\label{e:v1}
\end{align}

The steps that we now take to analyze this continuum eigenvalue problem parallel 
the steps that we will take to analyze the discrete eigenvalue problem~\eqref{v:v1}--\eqref{v:vM}.
First, we describe the
solutions of \eqref{e:vx} as linear combinations of solutions of the form
$e^{-zx}$ where
\begin{equation} \label{vzrel}
0 = \Lambda - \kappa z - \eps z^2. 
\end{equation}
It is convenient to write a general solution of \eqref{e:vx} in terms of the two
solutions $z_\pm$ of \eqref{vzrel} as 
\[
v(x)=c_+e^{z_+(1-x)} + c_- e^{z_-(1-x)}.
\]
Then a nonzero solution can satisfy 
the boundary conditions \eqref{e:v0} and \eqref{e:v1} 
if and only if 
\begin{equation} \label{e:detx}
0 = \left| \begin{matrix}f(z_+) & f(z_-)\\ g(z_+)& g(z_-)\end{matrix}\right| ,
\end{equation}
\begin{equation}
 f(z) = e^z(1 +\eps\kappa)+ \eps^2 z, \qquad
 g(z) = \kappa+\eps - \kappa \left(\frac{e^z-1}z\right) +\eps z \left(1+\frac{\eps^2}{1+\eps\kappa}\right).
\end{equation}

So far this is an exact treatment of the eigenvalue problem 
\eqref{e:vx}--\eqref{e:v1}.
But now we approximate, noting that 
the two solutions $z_\pm$ of \eqref{vzrel} satisfy
\begin{equation}\label{zzrel}
z_+ + z_- = -\frac\kk\eps, \quad z_+z_- = -\frac\Lambda\eps,\qquad
z_+ \approx \frac{\Lambda}{\kappa}, \quad z_- \approx -\frac\kappa\eps -\frac\Lambda\kappa,
\end{equation}
for small $\eps$.  For $\kappa>0$ 
and complex $z_+$ both of order $O(1)$, 
we neglect the exponentially small term $e^{z_-}=O(e^{-\kk/\eps})$
and keep only
the leading order terms in the other entries of the matrix in \eqref{e:detx},
expressing $z_-$ in terms of $z_+$ using the relations in \eqref{zzrel}.
Thus, writing $z=z_+$ we make the approximations
\begin{align}
&f(z_+) \approx e^z, \qquad\qquad\qquad\quad f(z_-) \approx \eps^2z_- \approx -\eps\kappa,
\\
&g(z_+) \approx \frac{\kk}z(z+1-e^z), \qquad
g(z_-) \approx (\kk+\eps z_-) \left(1+\frac1{z_-}\right)\approx  -\eps z.
\end{align}
Multiplying the second row of \eqref{e:detx} by $z/\kk$ and dividing the second column by $-\eps\kk$,
we find that the equation determining eigenvalues 
approximately takes the form $Q(z;\kappa)=0$, where
\begin{equation}
Q(z;\kappa):=
 \left| \begin{matrix} e^z & 1
\\ 
z+1-e^z & z^2/\kk^2
\end{matrix}\right| = e^z\left(1+\frac{z^2}{\kappa^2}\right)-(1+z).
\end{equation}

It is exactly this function $Q$ that we will find
responsible for the appearance of unstable eigenvalues
for the discrete problem~\eqref{v:v1}--\eqref{v:vM} in the limit of large $M$.
We analyze the complex roots of $Q$ in depth in Section~\ref{s:Q}.
It turns out that a complex-conjugate pair of roots $z$
emerges into the right half plane $\re z>0$ as $\kk$
increases past each value $\kk=\kk_j^0$ of an infinite sequence.
The values $\kk_j^0$ will be seen to be the limiting values of
$\kk_j(M)$ as appear in the statement of Theorem~\ref{thm:eigsB},
in the limit $M\to\infty$.

In the remainder of this paper, we carry out the 
proof of Theorem~\ref{thm:eigsB} by performing an
analogous analysis for the discrete eigenvalue problem,
including rigorous estimates for all the error terms.
For brevity's sake, we forgo the 
formulation and rigorous demonstration of 
results analogous to Theorems~\ref{thm:eigsB} and \ref{thm:hopf}
for the (parabolic) continuum model \eqref{e:Dtu}--\eqref{e:Mxcons}
described in this section.  
It should be evident from our analysis, though, 
that Hopf bifurcation occurs for this model in a similar way.

\section{Reformulation of the eigenvalue equation}\label{s:reform}

\subsection{The difference equation}
The eigenvalue equations \eqref{v:vell} for $2\leq\ell\leq N$ 
comprise a family of second order difference equations. 
These difference equations have solutions of the form
\begin{equation}
 V_\ell = c\vp^{M-\ell} \,, \quad 1\leq\ell\leq M,
\label{e:Vpow}
\end{equation}
whenever 
\begin{equation}
\lambda = K(\vp -1) + (\vp - 2 + \vp\inv)\,.
\label{d:lambdavp1}
\end{equation}
which we can rewrite using $A=K+1$ as 
\begin{equation}
\lambda +A+1 = A\vp + \vp\inv\,,
\label{d:lambdavp2}
\end{equation}
or as
\begin{equation}
A\vp^2 -(\lambda+A+1)\vp + 1 = 0.
\label{d:lambdavp}
\end{equation}
We take decreasing powers in \eqref{e:Vpow} for reasons of scaling explained below.

We can then ``connect" the values of $V_{1}$ and $V_{M}$ by means of
a transition matrix depending on two constants (for each value of $\lambda$).
More precisely, any solution of \eqref{v:vell} takes the form
\begin{equation}
V_{\ell}=
 c_{1}\left( \vp_{1}\right) ^{M-\ell}
+c_{2}\left( \vp_{2}\right) ^{M-\ell}\,, 
\quad 1\leq\ell\leq M, \label{A1E4}%
\end{equation}
whenever $\vp_{1}$ and $\vp_{2}$ are distinct roots of \eqref{d:lambdavp}.
Evidently the two roots are always related by 
$\vp_1\vp_2 = 1/A$, and for $\lambda=0$ the roots are 
$\vp_1=1$ and $\vp_2 = 1/A$.  

The roots are distinct except when $\vp_1 = \pm A^{-1/2}$,
which corresponds to 
\begin{equation}\label{e:spurious1}
\lambda=-1-A\pm2\sqrt A.
\end{equation}
For small $K>0$, we note that this becomes
\begin{equation}
\lambda = -2-K\pm 2\sqrt{1+K}\approx
\begin{cases} -K^2/4 & \text{for}\  +\,,\cr -4-2K & \text{for}\  - .
\end{cases}
\end{equation}

The roots $\vp_1$, $\vp_2$ are naturally functions of $\lambda$.
However, it will be more convenient to recast the 
eigenvalue equations in terms of the variable $\vp$ and regard
$\lambda$ as a function of $\vp$, given by the 
following equation equivalent to \eqref{d:lambdavp2}:
\begin{equation}
\lambda = \left(A - \vp\inv\right)(\vp-1)\,.
\label{dd:lambdavp}
\end{equation}
Except when $\vp=\pm A^{-1/2}$ (which will generate spurious roots below),
corresponding to $M$ eigenvalues $\lambda$ there should exist
$2M$ roots $\vp$ of the relevant equations, 
which occur in pairs $\vp$, $1/(A\vp)$ that produce the same $\lambda$.

\begin{remark}\label{r:ellipse}
We note that by \eqref{d:lambdavp2}, 
values of $\vp$ on the unit circle, with $\vp=e^{is}$ for $s$ real,
produce values of $\lambda$ on an ellipse with 
\begin{equation}\label{e:ellipse}
\lambda = (2+K)(-1+\cos s) + iK \sin s
\end{equation}
This ellipse lies in the left half plane and passes through $\lambda=0$. 
In numerical computations such as those reported in Fig.~\ref{f:evalsM100K3},
almost all the eigenvalues lie near this ellipse.
By consequence we will expect to find most roots satisfying 
$|\vp_1|\approx 1$ and $|\vp_2|\approx 1/A<1$, with $|\vp_2^M|$
extremely small. 
(This is the basic reason for the form we took in \eqref{e:Vpow}.)
The possibility of transition to instability will depend upon the 
deviation of roots $\vp_1$ from this ellipse in the vicinity 
$\vp\approx1$ where $\lambda\approx0$.
\end{remark}

\subsection{Reduction to a \texorpdfstring{$2\times2$}{2x2} determinant}

We now use the expression \eqref{A1E4} to write the ``boundary conditions'' for $V$,
that correspond to the equations for $V_M$ and $V_1$ in \eqref{v:vM} and \eqref{v:v1}
respectively. Using the fact that \eqref{d:lambdavp2} holds for both $\vp_1$ and $\vp_2$,
after some computation we find that these equations take the following form:
\begin{align}
0 &= c_{1}\left(  A\varphi_{1}^{M-1}+1-\varphi_{1}^{-1}\right)  +c_{2}\left(
A\varphi_{2}^{M-1}+1-\varphi_{2}^{-1}\right)  \,,
\label{S1E6}%
\\[5pt]
0  &= c_{1}\left(  -A\varphi_{1}^{M}-AM\varphi_{1}^{M-1}-K\frac{\varphi
_{1}^{M}-\varphi_{1}}{\varphi_{1}-1}+KM+1\right)  \nonumber\\
&  +c_{2}\left(  -A\varphi_{2}^{M}-AM\varphi_{2}^{M-1}-K\frac{\varphi_{2}%
^{M}-\varphi_{2}}{\varphi_{2}-1}+KM+1\right)  \,.
\label{S1E7}%
\end{align}
Except in the degenerate cases when $\vp_1=\vp_2=\pm A^{-1/2}$ and \eqref{e:spurious1} 
holds, the eigenvalue problem in \eqref{v:v1}--\eqref{v:vM} is therefore equivalent
to the vanishing of a determinant:
\begin{equation}
\delta\left(  \varphi\right)  =
\left\vert \begin{array} [c]{cc}%
f\left(  \varphi_{1}\right)  & f\left(  \varphi_{2}\right) \\
g\left(  \varphi_{1}\right)  & g\left(  \varphi_{2}\right)
\end{array} \right\vert =0 \,, 
\,,\label{e:del}%
\end{equation}
where $\vp_1=\vp$ and $\vp_2=1/(A\vp)$, and the functions $f$, $g$ are given by %
\begin{align*}
f\left(  \varphi\right)  &=A\varphi^{M-1}+1-\varphi^{-1} \,,\\%
g\left(  \varphi\right)  &=-A\varphi^{M}-AM\varphi^{M-1}-K\frac{\varphi
^{M}-\varphi}{\varphi-1}+KM+1 
\\ & = -\varphi^M\left(\frac{AM}\varphi+\frac{K}{\varphi-1} +A\right)
+\frac{K\varphi}{\varphi-1}+KM+1 \,.
\end{align*}
The function $\delta$ depends on $M$ and $K$, but this dependence
will not be written explicitly for simplicity.
We note the general root-exchange symmetry 
\begin{equation}\label{root:sym}
\delta\left(\frac 1{A\vp}\right) =-\delta(\vp) .
\end{equation}
Because $\lambda=0$ is an eigenvalue we also know that $\delta$
has roots at $\vp=1$ and $1/A$.
Note that $\delta(\pm A^{-1/2})=0$ due to dependence of the columns,
but these roots are spurious, unless double, as we now discuss.

{\em The degenerate case.}
In the cases of \eqref{e:spurious1} when the two roots of 
\eqref{d:lambdavp} coincide at $\vp=\vp_1=\pm A^{-1/2}=\vp_2$, 
one checks that the difference equation \eqref{v:vell} has the general solution
\begin{equation}
V_{\ell}=
 \hat c_{1} \vp^{M-\ell}
+ \hat c_{2} (M-\ell) \vp^{M-\ell-1}\,, 
\quad 1\leq\ell\leq M, \label{esp:vell}%
\end{equation}
by the expedient of replacing $c_1$, $c_2$ in \eqref{A1E4} with
\[
c_1 = \hat c_1 - \frac{\hat c_2}{\vp_2-\vp_1} , 
\qquad
c_2 = \frac{\hat c_2}{\vp_2-\vp_1} ,
\]
and taking $\vp_1\to \pm A^{-1/2}$. Doing the same with \eqref{S1E6}--\eqref{S1E7},
we see that the eigenvalue condition \eqref{e:del} is replaced by the condition 
\begin{equation}
\hat\delta\left(  \varphi\right)  =\left\vert
\begin{array}
[c]{cc}%
f\left(  \varphi\right)  & f'\left(  \varphi\right) \\
g\left(  \varphi\right)  & g'\left(  \varphi\right)
\end{array}
\right\vert =0  
\qquad\mbox{at $\vp=\pm A^{-1/2}$.}
\label{esp:del}%
\end{equation}
This is equivalent to the condition $\delta'(\vp)=0$
because one finds $\delta'(\vp)=-2\hat\delta(\vp)$ at these points. 

\begin{remark} 
In order to characterize Hopf bifurcation, we will use the fact
that when $\vp\ne\pm A^{-1/2}$, 1 or $A\inv$, 
$\vp$ is a simple root of $\delta(\vp)$ if and only if 
$\lambda=(A-\vp\inv)(\vp-1)$ is a simple eigenvalue of $B$.
See Lemma~\ref{lem:simpleBF} and its proof in Section~\ref{s:eigsB}.
\end{remark}

\subsection{Sorting terms and removing singularities}
For convenience in analysis, we sort the terms in \eqref{e:del}
according to $M$th powers of $\vp$ and $A$.  Note that
\begin{align}
f(\vp_2) &= (1-A\vp) + (A\vp)^{-M} A^2\vp \,, 
\\
g(\vp_2) &= \left(\frac{K}{1- A\vp} + KM+1 \right)
- (A\vp)^{-M}\left(MA^2\vp + \frac{KA\vp}{1-A\vp} + A\right).
\end{align}
In order to remove singularities, we multiply \eqref{e:del} by
$\vp(\vp-1)(1-A\vp)$. Define
\begin{equation}
F(\vp) := \delta(\vp)\cdot \vp(\vp-1)(1-A\vp) = \left\vert
\begin{array}[c]{cc}
f_1 & f_2 \\ g_1 & g_2
\end{array}
\right\vert\,,
\end{equation}
where 
\begin{align*}
f_1 &= f(\vp)\cdot \vp = (\vp -1) + \vp^M A \,, \\
f_2 &= f(\vp_2) = (1-A\vp) + (A\vp)^{-M} A^2\vp \,, \\
g_1 &= g(\vp)\cdot \vp(\vp-1)(1-A\vp) \quad = G_1 - \vp^M G_2 \,,\\
g_2 &= g(\vp_2)\cdot (\vp-1)(1-A\vp) \quad = H_1 - (A\vp)^{-M} H_2 \,,
\end{align*}
with the definitions 
\begin{align}
   G_1 &= \vp (1-A\vp) (K\vp + (KM+1)(\vp-1) )    \,, \label{d:G1}\\ 
  G_2 &= (1-A\vp) ( AM(\vp-1)+  K\vp+A\vp(\vp-1) ) \,,  \label{d:G2}\\ 
  H_1 &= (\vp-1) (K + (KM+1)(1-A\vp) )           \,,  \label{d:H1}\\ 
  H_2 &= (\vp-1) ( (A+MA^2\vp)(1-A\vp) + KA\vp ) \,.  \label{d:H2}
\end{align}
By consequence we have the sorted representation
\begin{equation}
\boxed{
F(\vp) = -P_1+ \vp^M P_2 + A^{-M} R_1+ (A\vp)^{-M} R_2,
}
\label{F:sort}
\end{equation}
where
\begin{align*}
-P_1 &= 
\left| \begin{matrix}
\vp - 1 & 1-A\vp \\
G_1 & H_1
\end{matrix}\right\vert
\,, \qquad 
P_2 = 
\left| \begin{matrix}
A & 1-A\vp \\
-G_2 & H_1
\end{matrix}\right\vert
\,, \\
R_1 &= 
\left| \begin{matrix}
A & A^2\vp \\
-G_2 & -H_2
\end{matrix}\right\vert
\,, \qquad 
R_2 = 
\left| \begin{matrix}
\vp -1 & A^2\vp \\
G_1 & - H_2 
\end{matrix}\right\vert
\,. 
\end{align*}
Observe that $F$ has a pole at $\vp=0$ of order $M$,
with $F(\vp)\sim -A^{1-M}\vp^{-M}$, 
because $R_2 = H_2= -A$ at the origin.
And for $|\vp|\to\infty$ we find that
\[
F(\vp) \sim \vp^M P_2 \sim \vp^{M}A\vp G_2\sim A^3\vp^{M+4}.  
\]
Consequently $F$ must have exactly $2M+4$ zeros, counting multiplicities.

We may summarize the situation as follows.
\begin{proposition}\label{p:lambdaphi}
A complex number
$\lambda$ is an eigenvalue of $B$
if and only if \eqref{dd:lambdavp} holds for some 
pair $\vp$, $1/A\vp$ satisfying
\begin{equation} \label{e:Feq0}
F(\vp)=0,
\end{equation}
except in the two cases 
$\lambda=-1-A\pm2\sqrt A$. 
In these cases, $\lambda$ is an eigenvalue if and only if 
\begin{equation}
F(\vp)=F'(\vp)=0 \qquad\mbox{at $\vp = \pm A^{-1/2}$.}
\label{esp:Fp}
\end{equation}
\end{proposition}
Of the $2M+4$ roots of $F$, four are spurious,  counting
$\vp=\pm A^{-1/2}$, $\vp=1$ and $\vp=1/A$ once each,
coming from the dependence of the columns in \eqref{e:del}
and the factors used to remove singularities from $\delta$.  

The polynomial 
$\vp^M F(\vp)$ of degree $2M+4$ is divisible by the factor
\begin{equation}
S(\vp) = (\vp-1)(A\vp-1)(A\vp^2-1) ,
\label{d:Svp}
\end{equation}
and the remaining $2M$ roots of $\vp^M F(\vp)/S(\vp)$
correspond in pairs $\vp$, $1/(A\vp)$ to the $M$ eigenvalues of $B$. 
The values $\vp=1$ and $1/A$, are (at least) double roots of $F$
because they were already roots of $\delta$, and correspond 
to the simple eigenvalue $\lambda=0$.
Concerning other roots of $F$, we have the following result
whose proof we defer to Section~\ref{s:eigsB}.

\begin{lemma}\label{lem:simpleBF}
Suppose $S(\vp)\ne0$ 
and $\lambda= (A-\vp\inv)(\vp-1)$. 
Then $\vp$ is a simple root of $F$ if and only if 
$\lambda$ is a simple eigenvalue of $B$.
\end{lemma}

\section{Formal approximation}\label{s:formal}

Before we begin a rigorous analysis of the zeros of $F(\vp)$, 
we treat the problem approximately in the limit of large $M$ 
to gain insight. Numerical experimentation suggests that
we can expect to find most solutions of \eqref{e:del}
to satisfy $|\varphi_1|\approx 1$, and $|\varphi_2|\approx A^{-1}<1$,
with $A^{-M}$ extremely small.

Thus we neglect the terms containing $A^{-M}$ in \eqref{F:sort} and 
study the zeros of 
\begin{equation}
\label{d:F0}
\boxed{
F_0(\vp) := -P_1(\vp)+\vp^M P_2(\vp) \,.
}
\end{equation}
For any such zero, evidently 
\begin{equation}
\vp^M = \frac{P_1(\vp)}{P_2(\vp)} \,,
\label{e:vpMrat}
\end{equation}
unless both numerator and denominator vanish. 
The right-hand side  is a ratio of polynomials of low degree,
while for large $M$, the function $\vp\mapsto \vp^M$ expands a small region
about any $M$th root of unity $e^{2\pi ik/M}$ 
to cover a large part of the complex plane. 
Roughly, then, we can expect \eqref{e:vpMrat} to have a solution near 
each $M$th root of unity. 
These should then provide eigenvalues 
spread out around the ellipse in \eqref{e:ellipse}.

We focus next on looking for imaginary roots $\vp\approx1$.
We change variables from $\vp$ to $z=M(\vp-1)$, %
noting that
\begin{equation}
\vp^M = \left(1+\frac{z}{M} \right)^M \to e^z 
\quad\text{ as }\ \ M\to\infty.
\end{equation}
With these relations we have 
\[
A\vp -1 = K+A(\vp-1) = K+\frac{Az}{M}
= K + \frac{z}{M} +\frac{Kz}{M}\,,
\]
and we find from \eqref{d:G1}--\eqref{d:H1} the exact expressions 
\begin{align}
G_1 &= -\left( 1 + \frac{z}{M} \right)
\left(K+\frac{Az}{M}\right)
\left( K
\left( 1 + \frac{z}{M} \right)
+ \left(K+\frac{1}{M}\right) z
\right),
\nonumber \\
G_2 &= -\left(K+\frac{Az}{M}\right) Az
- \left(K+\frac{Az}{M}\right)^2 
\left( 1 + \frac{z}{M} \right),
\nonumber\\
H_1 &= z \left( \frac{K}{M}- \left(K+\frac{1}{M}\right) 
\left(K+\frac{Az}{M}\right)  \right) 
\nonumber\\ &= 
  z \left(- K^2 -  \left(K+\frac{1}{M}\right) \frac{Az}{M}\right)\,.
\label{e:H1z}
\end{align}
It turns out to be appropriate to require $K$ is small while $KM$ is large.
Somewhat more precisely, we ask that
\begin{equation}
K = O(\eps) \quad\text{as}\ \ \eps:=\frac{1}{\sqrt{M}} \to0.
\end{equation}
Then we get the approximate relations
\begin{align}
G_1 &= -K^2 (1+z) + O(\eps^{3}) \,,
\nonumber\\
G_2 &= -K(1+K)z - \frac{z^2}{M} - K^2 
+O(\eps^{3})\,,
\nonumber\\
H_1 &= - K^2 z - \frac{Kz^2}{M} + O(\eps^4)\,.
\end{align}
By consequence, we find that 
\begin{align*}
P_1 &= K^3 (1+z) + O(\eps^4),
\\
P_2 &=
(1+K)\left(-K^2 z - \frac{Kz^2}{M} \right)+O(\eps^4)
\\ 
&\quad  +\left(K+\frac{z}{M}\right)
\left(K(1+K)z+\frac{z^2}{M} +K^2\right)
\\
&=   \frac{Kz^2}{M}  +K^3 +O(\eps^4) .
\end{align*}
If we suppose $K\sim \kk/\sqrt{M}$ as $M\to\infty$, then
\begin{equation}
K^{-3} F_0\left(1+\frac{z}{M}\right) \to  Q(z;\kk)
\label{d:Phiza}
\end{equation}
where 
\begin{equation}
\boxed{
Q(z;\kk) := e^z\left(1+\frac{z^2}{\kk^2} \right)-(1+z)
}
\label{d:Qa}
\end{equation}

The complex roots of $Q$ provide an approximation for 
roots of $F_0(\vp)$ when $M$ is large. These approximate 
eigenvalues $\lambda$ of \eqref{v:v1}--\eqref{v:vM}
through \eqref{d:lambdavp1}, which may be written directly 
in terms of $z$ as 
\begin{equation}
\lambda = \frac{Kz}{M} + \frac{z^2/M^2}{1+z/M} = 
\frac{Kz}{M} + O(\eps^4).
\label{e:lambdaz}
\end{equation}
Thus purely imaginary roots of $Q$ approximate 
eigenvalues $\lambda$ near the imaginary axis,
and roots of $Q$ in the right half plane $\re z>0$
should approximate eigenvalues satisfying $\re\lambda>0$.

\section{Analysis of roots of \texorpdfstring{$Q$}{Q}}\label{s:Q}

In this section we establish basic properties of the roots $z$
of $Q(z;\kk)$ as defined in \eqref{d:Qa}. 
This will serve as the foundation to analyze the roots of $F_0$
and ultimately those of $F$, in subsequent sections.

Purely imaginary roots $z=it$ of $Q$ occur whenever
\begin{equation}
e^{it}\left(1-\frac{t^2}{\kk^2}\right) = 1+it\,.
\label{e:Qcond}
\end{equation}
Matching real parts demands that if $t\ne0$ then
$\sec t=1-\frac{t^2}{\kappa^2}<1$ hence $\sec t\le-1$.
Matching also the ratio of imaginary to real parts,
one finds that \eqref{e:Qcond} holds if and only if
\begin{equation}
\cos t<0
\quad\mbox{and}\quad 
\tan t  =t,
\label{e:ttant}
\end{equation}
together with
\begin{equation}
\kk^2 = \frac{t^2}{1-\sec t} 
= \sqrt{1+t^2}-1.
\label{e:alphat}
\end{equation}
Each positive root of \eqref{e:ttant} 
provides a complex conjugate pair of imaginary roots $z=\pm it$ of $Q$.
Let $t_1<t_2<\ldots$ denote the increasing sequence of all these positive roots of \eqref{e:ttant}.
The smallest occurs for $t=t_1\approx 4.4934095$
(less than $\frac32 \pi \approx 4.71238898$). \
This corresponds to a critical value of $\kk$ given by 
\begin{equation}\label{e:kappacr}
\kk_{\rm cr} := (\sqrt{1+t_1^2}-1)^{1/2} \approx 1.89825.
\end{equation}
The roots $t_k$ approach $\frac32\pi+2\pi k$ from below as $k\to\infty$.
As $k$ increases, they correspond to larger values of $\kk^2$, 
hence larger values of $K$ for a fixed $M$.

In the rest of this section, we shall prove that non-real roots of $Q$ are always simple,
and purely imaginary roots must move into the right half plane $\re z>0$
as $\kk$ increases, where they must remain in a bounded region. 
By this result and \eqref{e:lambdaz}, when $\kk>\kk_{\rm cr}$
we can expect that for large enough $M$ with $K\sim \kk/\sqrt{M}$, 
there will be some eigenvalue $\lambda$ of \eqref{v:v1}--\eqref{v:vM}
in the right half plane, and when $0<\kk<\kk_{\rm cr}$
we can expect there will not.

\begin{lemma}
For any $\kk>0$, $Q$ has a double root $z=0$. All
other complex roots are non-real and simple.
\end{lemma}
\begin{proof}
Clearly $Q(0;\kk)=0$, and for real $z\ne0$ we have
$Q > e^z - 1-z >0$ by the convexity of $e^z$.
In general we compute 
\[
\D_zQ = Q + z + e^z(2z/\kk^2).
\]
The root $z=0$ is double because $0=Q=\D_zQ<\D^2_zQ$ at 0.
At a complex double root, on the other hand, necessarily
$e^z = -\kk^2/2$. This implies $z=r+i\pi k$ where
$e^r=\kk^2/2$ and $k$ is an odd integer.
Then, however, it follows
\[
0 = -2Q = z^2+\kk^2 + 2 + 2z = (r^2-\pi^2k^2 + 2e^r +2 + 2r) + i\pi k(2r+1),
\]
so $r=-\frac12$ and we infer
\[
\pi^2<\pi^2k^2 = \frac14+2e^{-1/2}+1 < 4,
\]
a contradiction.  Hence the nonzero roots of $Q$ are all non-real and simple.
\end{proof}

For the next result, let $t_0=0$ and recall that $t_1<t_2<\ldots$
denotes the sequence of positive roots of \eqref{e:ttant}.
\begin{lemma} \label{l:Qzeros}
The function $Q$ has exactly $k$ complex-conjugate pairs
of roots $z$ in the right half plane  $\re z>0$ 
if $\kk^2=\sqrt{1+t^2}-1$ with $t\in(t_k,t_{k+1}]$.
\end{lemma}
\begin{proof}
First, we claim that the imaginary roots of $Q$ always 
cross into the right half plane $\re z>0$ as $\kk$ increases. 
To see this, regard $w:=\kk^2$ as a complex
variable and note that $Q=0$ if and only if 
\begin{equation}\label{e:wfromz}
w = \frac{z^2}{(1+z)e^{-z}-1}\,.
\end{equation}
Because $\frac{d}{dz}(1+z)e^{-z} =-ze^{-z}$, we compute 
\begin{equation}\label{ee:dzdw0}
\frac{z}{w} \left(\frac{dz}{dw}\right)\inv = 2+ we^{-z} = 2 + \frac{w+z^2}{1+z}\,,
\end{equation}
by using the identity $Q=0$ to eliminate $e^{-z}$.
{
Multiplying by $\bar z|1+z|^2$, we find 
\begin{equation}\label{ee:dzdw}
\frac{|z|^2|1+z|^2}w 
\left(\frac{dz}{dw}\right)\inv = 
2\bar z|1+z|^2 + w(\bar z+\bar z^2) + z|z|^2 + |z|^4
\end{equation}
For $z=x+iy$ in the first quadrant,
the imaginary part of this expression is negative, which implies
\begin{equation}\label{e:imdzdk}
\im \frac{dz}{d\kk} >0\,.
\end{equation}
Furthermore, provided $y^2>\kk^2$ 
(which must be the case if $x=0$ by \eqref{e:alphat}),
the real part of \eqref{ee:dzdw} is larger than
$y^4-wy^2>0$, hence
\begin{equation}\label{e:redzdk}
\re \frac{dz}{d\kk} >0\,.
\end{equation}
} 
It follows from these computations that the roots $z=\pm it_k$ of $Q$ 
on the imaginary axis always pass 
into the right half plane as $\kk$ increases, with derivative $dz/d\kk$
remaining in the first quadrant. They can never escape
to infinity, because any roots of $Q$ in the right half plane must lie in
the bounded region where 
\[
1> | e^{-z} | = \left| \frac{\kk^2+z^2}{\kk^2(1+z)} \right|.
\]

To finish the proof, we show that if $\kk>0$ is small enough, 
then $Q$ has no roots with $\re z>0$. If $\kk\in(0,1)$,
any such root must satisfy
\[
\left|\frac{z^2}{\kk^2}\right| -1 \leq 
\left|\frac{z^2}{\kk^2}+1\right|  
= |(1+z)e^{-z}|  < 1+ |z|
< 1+\left|\frac z\kk\right|,
\]
and this implies $|z|<2\kk$.  Now it follows
\begin{align*}
\kk^2 Q &= e^z z^2 + \kk^2(e^z-1-z)  \\
&= z^2 + O(z^3) + \kk^2\left( \frac12 z^2 + O(z^3) \right)
\\& = z^2\left(1+O(\kk) \right).
\end{align*}
Therefore, for small enough $\kk>0$, $Q$ does not vanish when $\re z>0$.
\end{proof}

{\em Labeling the roots.}
Due to the results of the previous two lemmas, we may label all the non-real roots
of $Q$ that cross the imaginary axis and lie in the upper half plane $\im z>0$
by analytic functions $z = z_j^0(\kk)$, $j=1,2,\ldots$, defined for all $\kk>0$
according to the property that 
\begin{equation}\label{d:zk}
z_j^0(\kk)=it_j \quad\mbox{ when 
\ \ } \kk=\kk^0_j:= \left(\sqrt{1+t_j}-1\right)^{1/2}.
\end{equation}
Thus we can summarize as follows. 
\begin{lemma}\label{lem:Q0zj}
There are analytic curves $z_j^0\colon(0,\infty)\to\CC$, $j=1,2,\ldots$, 
satisfying \eqref{d:zk} and $\im z_j^0(\kk)>0$ for all $\kk>0$, such that 
when $\kk\in(\kk^0_k,\kk^0_{k+1}]$, 
the numbers $z_1^0(\kk)\ldots,z_k^0(\kk)$
comprise all the roots of $Q$ in the first quadrant.
Moreover, $dz_j^0/d\kk\ne0$ for all $\kk>0$, and 
\begin{equation} \label{e:dzk}
\re \frac{dz_j^0}{d\kk} >0 \quad\mbox{and}\quad
\im \frac{dz_j^0}{d\kk} >0 
\quad\mbox{for all $\kk\ge\kk_j^0$.}
\end{equation}
\end{lemma}

\begin{proof}
To show the curves $z_j^0$ are well defined and nondegenerate for all $\kappa>0$,
we note that according to standard continuation theory
for the ODE \eqref{ee:dzdw0}, a solution exists for real $w$ 
in a maximal interval $(w_-,\infty)\subset(0,\infty)$ for which 
$dz/dw$ remains bounded.  
It is not possible that $w_->0$, however, because the right-hand side
of \eqref{ee:dzdw0} cannot approach zero at the same time as 
\eqref{e:wfromz} holds with $w\to w_-$, for the following reason:
If \eqref{ee:dzdw0} vanishes, then
$0=2(1+z)+w+z^2$, hence $z=-1+i\tau$ with $\tau=\sqrt{1+w}>1$. 
But then \eqref{e:wfromz} implies
\[
0=w(1+z)e^{-z}-w-z^2 = i\tau((\tau^2-1)e^{1-i\tau}+2).
\]
This implies $\tau^2 = 1 - 2e^{i\tau-1}$, so necessarily $\sin\tau=0$
but also $1<\tau^2<1+2/e$, and this is impossible.
\end{proof}

\section{Analysis of roots of \texorpdfstring{$F_0$}{F0}}

In this section we locate all the roots of the polynomial 
$F_0=\vp^M P_2 -P_1$ in \eqref{d:F0} of degree $M+4$, 
to a rough approximation, provide bounds on roots that 
may correspond to unstable eigenvalues, and establish the 
convergence in \eqref{d:Phiza} in a precise sense.
Let $B(z,r)\subset\CC$ denote the closed disk 
with center $z\in\CC$ and radius $r>0$.
We fix a constant $\am>2$. (Actually, $\am=3$ suffices.)
Depending on some large $\beta>1$ 
(to be chosen in the proof of Theorem~\ref{thm:eigsB}), 
we presume throughout that 
\begin{equation}\label{h:KM}
\beta\inv\le K\sqrt{M}\le \beta.
\end{equation}

\subsection{Rough locations of all roots}
Locations of the $M+4$ roots of $F_0$ will be identified as follows.
We recall that the four values $\vp=1$, $A\inv$, $\pm A^{-1/2}$,
which comprise the roots of the polynomial 
\[
S(\vp)=(\vp-1)(A\vp-1)(A\vp^2-1)
\]
from \eqref{d:Svp},
are already known to be roots of the function $F$ that $F_0$ approximates.
Note that the three roots of $S(\vp)$ with $\vp\ne1$ satisfy
\begin{equation}\label{r:S0appr}
A\inv = 1-K+o(K), \quad
\pm A^{-1/2} = \pm\left(1-\frac K2\right) +o(K)\,.
\end{equation}

\begin{proposition}\label{p:F0roots}
Fix $\am>2$.  Then for any $\beta>1$ there exists $\ap_0>0$ and $M_0>0$
such that whenever $M>M_0$ and \eqref{h:KM} holds,
the polynomial $F_0$ has exactly:
\begin{itemize}
\item[(i)] one double root at $\vp=1$.
\item[(ii)] one simple root in each of the following disks 
of radius $r_K=K/8$:
\[
\mbox{$B(A\inv,r_K)$,  \quad
$B(A^{-1/2},r_K)$  , \quad
$B(-A^{-1/2},r_K)$  . \quad
}
\]
\item[(iii)] one simple root in $B(-M,1)$.
\item[(iv)] $M-2$ roots in the punctured annulus 
\[
D_a 
:= \left\{\vp: \vp\ne1 \mbox{ and } M^{-\am/M}<|\vp|<
(1-\frac {\ap_0} M)\inv \right\}.
\]
\end{itemize}
\end{proposition}

\begin{proof}
Recall $F_0(\vp)=\vp^MP_2-P_1$. where we can write
\begin{align} \label{e:P1-1}
P_1 &= (A\vp-1)^2\vp(K\vp + (KM+1)(\vp-1) )
\nonumber\\ & \qquad -(\vp-1)^2(K - (KM+1)(A\vp-1) )
\nonumber\\ &= K[ (A\vp-1)^2\vp^2 - (\vp-1)^2 ] + (KM+1)S(\vp)\,,
\end{align}
with $S(\vp)$ 
as in \eqref{d:Svp}, and 
\begin{align} \label{e:P2-1}
P_2 &= (A\vp-1)^2 [ K\vp + A(\vp-1)(M+\vp)]
\nonumber\\ & \qquad - 
A(\vp-1)[KM(A\vp-1)+A(\vp-1)] 
\nonumber\\ &=  \vp(A\vp-1)^3 -A^2(\vp-1)^2 + M(A\vp-1) A^2(\vp-1)^2
\,.
\end{align}

{\bf Step 1.}  First we establish (i). Note that $F_0(1)=0$, since
\[
P_1(1)=P_2(1)=K^3.
\]
Furthermore, $F_0'(1)=MK^3+P_2'(1)-P_1'(1)=0$ since
\begin{align*}
P_1'(1) &= (KM+1)K^2 + 2K^2(1+2K) = MK^3 + K^2(3+4K)\,,
\\
P_2'(1) &= 2AK^2 + K^2(K+A(M+1)) - AMK^2 = K^2(3+4K)\,.
\end{align*}
Hence $\vp=1$ is at least a double root. But one also checks
\begin{equation}\label{e:F0pp}
F_0''(1) = M(M+1)K^3 + 2A^2(KM-1)
\end{equation}
(e.g., by computer algebra) so $F_0''(1)\ne0$  when $KM\ge1$.

{\bf Step 2.}  Next we claim that the only roots of $F_0$ in the disk
$B(0,M^{-\am/M})$ are three as described in (ii).  
We can write 
\begin{equation}
-F_0(\vp)=P_1-\vp^M P_2 = (KM+1)S(\vp)+KS_1-\vp^M P_2\,,
\end{equation}
where
\begin{equation}\label{e:P1frac}
S_1 =  (A\vp-1)^2\vp^2 - (\vp-1)^2 \,. 
\end{equation}
It suffices to show that 
for all $\vp$ in $B(0,M^{-\am/M})$ outside the balls listed in (ii),
\begin{equation}\label{b:step2}
\Delta_0:=KM|S(\vp)|- K|S_1| - |\vp^M P_2| >0\,,
\end{equation}
for $M$ large enough. 
For then our claim follows from Rouch\'e's theorem,
since each of the balls in (ii) contains one simple root of $S$.

Observe that $|P_2|\le CM$ for $|\vp|\le1$, therefore
\begin{equation}
|\vp^M P_2|  \le C M^{1-\am}.
\end{equation}
(Here and below $C$ denotes a generic constant which may 
depend on $\beta$ and $\am$ but is independent of $M$ and $K$, 
whose value may change from instance to instance.) 
To complete the proof of \eqref{b:step2}, we consider three sub-cases:
\[
\mbox{(a) $\re\vp<0$;\quad (b) $\re \vp>0$ and $|\vp-1|>2K$;\quad
(c) $|\vp-1|\le 2K$.}
\]

In case (a), for each $\hat\vp\in\{1,A\inv,A^{-1/2}\}$
(i.e., for each positive root of $S$), 
we have $1-K<|\vp-\hat\vp|< 2$, therefore
\begin{equation}
|S(\vp)|\ge A^2(1-K)^3 |\vp+A^{-1/2}| \quad\mbox{and}\quad
|S_1| \le 4(A^2+1)<8A^2  \,.
\end{equation}
Because $\am>2$ and $K^2M\ge \beta^{-2}$ 
it follows that for $|\vp+A^{-1/2}|\ge r_K=K/8$,
with $M$ large enough we have
\begin{align} 
\Delta_0 
&\ge 
\frac{K^2M A^2}{16} - \frac{8A^2}{\beta\sqrt{M}}- \frac C{M^{\am-1}}
\ge \frac{\beta^{-2}}{20}>0.
\label{e:F0casea}
\end{align}
(We could replace $r_K$ by say $20/M$ here, but we have no need.)

In case (b), each positive root of $S$ satisfies $|1-\hat\vp|\le K$,  
hence 
\[
\frac12|\vp-1|\le |\vp-1|-K\le |\vp-\hat\vp|\le |\vp-1|+K <2|\vp-1| \,.
\]
Consequently
\[
|S_1| \le A^2 |\vp-A\inv|^2+|\vp-1|^2 \le 5A^2|\vp-1|^2
\]
and (for $K<\frac12$)
\[
|S(\vp)|\ge \frac14 A^2(1-K) |\vp-1|^3 \ge \frac14 A^2 K |\vp-1|^2\,.
\]
Therefore as in \eqref{e:F0casea} we get
\begin{align}
\Delta_0 &\ge
\left(\frac{K^2MA^2}8 - \frac{5A^2}{\beta\sqrt{M}}\right)|\vp-1|^2 
- \frac C{M^{\am-1}}
\ge \frac{\beta^{-4}}{20M} >0
\label{e:F0caseb}
\end{align}
for $M$ large enough depending on $\beta$ and $\am$.

In case (c), we have 
\begin{align*}
|S(\vp)|&\ge A^2(1-3K)\min_{\hat\vp} |\vp-\hat\vp|^3\,,
\\
|S_1| &\le 2A^2\max_{\hat\vp}|\vp-\hat\vp|^2 \le 10A^2K^2 \,,
\end{align*}
with min and max taken over positive roots of $S$. Therefore
for $M$ large, when $|\vp-\hat\vp|\ge r_K = K/8$ (chosen to separate the roots)
we find
\begin{align}
|\Delta_0| &\ge 
\frac{\sqrt{M}A^2}\beta  \left(
\frac{r_K^3}2 - \frac{10K^2}{M}
\right) - \frac C{M^{\am-1}} \ge 
\frac cM \,,
\label{e:F0casec}
\end{align}
for some $c>0$ depending on $\beta$. 

This finishes the proof of \eqref{b:step2}.  The conclusion in (ii) 
now follows, and also the fact that 
$F_0$ has no other roots in $B(0,M^{-\am/M})$.

{\bf Step 3.} Next we show that $F_0$ has no roots
satisfying 
\begin{equation}\label{b:outer1}
|\vp|\inv < 1- \frac {\ap_0} M \quad\mbox{and}\quad |M+\vp|\ge1\,,
\end{equation}
for large enough ${\ap_0} $ depending on $\beta$, 
and deduce (iii) and (iv).
The estimates in \eqref{b:outer1} imply
\begin{equation} \label{b:outer2}
|1-\vp\inv| \ge 1-|\vp|\inv> \frac {\ap_0} M \quad\mbox{and}\quad
|M+\vp||\vp| \ge \frac{M}2 \vee|\vp|\,.
\end{equation}
Observe 
\[
\vp^{-M}F_0(\vp) = P_2 - \vp^{-M}P_1 = S_4+S_3 -\vp^{-M}P_1\,, \] 
where
\begin{align}\label{d:Q2}
S_4 &:= (M+\vp)(A\vp-1)A^2(\vp-1)^2 \,,
\\[5pt]
S_3 &:= 
 \vp(A\vp-1)^3 -A^2(\vp-1)^2 - \vp(A\vp-1) A^2(\vp-1)^2
 \nonumber\\
 &= \vp (A\vp-1)(K^2+2KA(\vp-1))-A^2(\vp-1)^2 \,.
\end{align}
(To get this last, expand 
$(A\vp-1)^2=(K+A(\vp-1))^2$
and cancel a term.)

We now show the ratios $S_3/S_4$ and $\vp^{-M}P_1/S_4$ are uniformly small
for $\vp$ satisfying \eqref{b:outer1}, by estimating six terms as follows:

(a) The first term of the ratio $S_3/S_4$ is bounded 
using \eqref{b:outer2} as follows:
\begin{align}
\left| \frac{\vp(A\vp-1)K^2}{S_4} \right| =
\frac{ K^2}{|M+\vp||\vp| A^2 |1-\vp\inv|^2} \le 
\frac{2K^2} M \frac{M^2}{\ap_0^2} \le \frac{2\beta^2}{\ap_0^2}
\end{align}

(b) To bound the next term in $S_3/S_4$, observe
\begin{align}
S_*:=\left| \frac{ \vp 2K A(\vp-1) }{(M+\vp)A^2(\vp-1)^2} \right|
\le 
\frac{2K}
{|M+\vp| (1-|\vp|\inv)
}
\end{align}
For $|\vp|> M/2$, since $|M+\vp|\ge1$, for $M>4$ we have 
\begin{equation}
S_* \le \frac{2K}{1-2/M} < 4K \le \frac{4\beta}{\sqrt{M}} \,,
\end{equation}
while for $|\vp|\le M/2$ we have $|M+\vp|>M/2$ and
infer from \eqref{b:outer2} that 
\begin{equation}
S_* \le \frac{4K}{M} \frac M{\ap_0}  
\le 4K\le \frac{4\beta}{\sqrt{M}}\,.
\end{equation}

(c) The last term in the ratio $S_3/S_4$ satisfies the bound
\begin{align}
\left| \frac{A^2(\vp-1)^2 }{S_4} \right| =
\frac{1}{ |M+\vp||\vp| (K+1-|\vp|\inv)} \le \frac{2}{M}\frac 1K\le \frac{2\beta}{\sqrt{M}}\,.
\end{align}

(d) The terms in $\vp^{-M}P_1/S_4$ are estimated as follows.
By \eqref{b:outer2}, 
\[
|M+\vp||\vp-1|\ge \frac {\ap_0} 2 \,. 
\]
Further, $A\vp^2-1 = A(\vp-A^{-1/2})(\vp+A^{-1/2})$ and 
\begin{equation} \label{b:Amhalf}
A|\vp-A^{-1/2}|\le A|\vp-1|+A(1-A^{-1/2})< A|\vp-1|+K \,.
\end{equation}
Therefore, since $|\vp+A^{-1/2}|<2|\vp|$ and recalling $|M+\vp|\inv \le 2|\vp|/M$,
\begin{align} 
\left|
\frac{S(\vp)}{S_4}
\right|
&\le 
\frac{2A|\vp|(A|\vp-1|+K)}{|M+\vp||\vp-1|A^2}
\le \frac{4|\vp|^2}{M} + \frac{4K|\vp| }{\ap_0}
\,.
\end{align} 
Hence, since $KM+1<2KM$, the last term in $\vp^{-M}P_1/S_4$ is bounded by 
\begin{align} 
\left|
\frac{2KM S(\vp)}{\vp^M S_4}
\right|
\le 
\frac{8K}{|\vp|^{M-2}} + \frac{8K^2M}{|\vp|^{M-1} \ap_0}
\le \frac{8\beta}{\sqrt M} + \frac{8\beta^2}{\ap_0}\,.
\end{align}

(e) For the next term in $\vp^{-M}P_1/S_4$, we have the bound
\begin{align}
\left|
\frac{K(A\vp-1)^2\vp^2}{\vp^M S_4}
\right|
&\le \frac{K(K+A|\vp-1|)|\vp|^{2-M}}{|M+\vp||\vp-1|^2A^2}
\nonumber\\
&\le \frac{2K^2M}{\ap_0^2} + \frac{2K}{\ap_0} \le \frac{2\beta^2}{\ap_0^2} + \frac{2\beta}{\ap_0\sqrt{M}} \,.
\end{align}

(f) Lastly we have the bound
\begin{align}
\left|
\frac{K(\vp-1)^2}{\vp^M S_4}
\right| \le
\frac{K|\vp|^{-M}}{|M+\vp||\vp|(K+1-|\vp|\inv)} \le  \frac2M\,.
\end{align}

Assembling the estimates in (a)-(f), we conclude that 
if $\ap_0\ge \ap_0(\beta)$ and $M\ge M_0(\beta)$, then
\begin{equation}\label{b:3M}
\frac{ |\vp^{-M}F_0(\vp)-S_4|}{|S_4|}  
<\frac12
\end{equation}
for all $\vp$ satisfying \eqref{b:outer1}.
Part (iii) now follows by Rouch\'e's theorem
since $S_4$ has only one simple zero at 
$\vp=-M$ inside $B(-M,1)$.  Part (iv) follows since we have
shown that $F_0$ has exactly 6 roots (counting multiplicity)
in the complement of the punctured annulus $D_a$.
\end{proof}

We record here several estimates that follow from the proof above.
\begin{corollary}\label{c:F0ests}
Under the conditions of Proposition~\ref{p:F0roots}, we have the following
estimates, for some $c>0$ depending on $\beta$:
\begin{itemize}
\item[(i)]  \quad $\displaystyle  
|F_0(\vp)| \ge 
\frac{\beta^{-2}}{20} 
$
\quad if \ \ $|\vp+A^{-1/2}|=r_K$.
\item[(ii)]  \quad $\displaystyle  
|F_0(\vp)| \ge \frac{c}M 
$
\quad\  if \ \ ${|\vp-A\inv|=r_K}$ 
or\ \  ${|\vp-A^{-1/2}|=r_K}$.
\item[(iii)]  \quad $\displaystyle  
|F_0(\vp)| \ge \frac{c}M 
$
\quad\ if \ \ ${|\vp|=M^{-\am/M}}$.
\item[(iv)]
\quad $\displaystyle 
|\vp^{-M}F_0(\vp)| \ge \frac 12M^3 $
\quad if \ \ ${|\vp+M|=1}$.
\item[(v)]  \quad $\displaystyle  
|F_0(\vp)| \ge  \frac{K\ap_0^2}{2M^2} $
\quad\ if \ \ $\displaystyle {|\vp|\inv=1-\frac{\ap_0}M}$.
\end{itemize}
\end{corollary}
\begin{proof}
Part (i) follows from \eqref{e:F0casea} in case (a) of Step 2,
because $|F_0(\vp)|\ge\Delta_0$. 
Similarly, part (ii) follows from \eqref{e:F0casec} in case (c) of Step 2,
and part (iii) follows from all cases of Step 2.
To infer part (iv), note that \eqref{b:3M} of Step 3 implies
that for $|\vp+M|=1$ we have
\begin{equation}
2|\vp^{-M}F_0(\vp)| \ge 
|S_4| = A^2|A\vp-1||\vp-1|^2 \ge M^3 \,, 
\end{equation}
because $|A\vp-1|\ge A|\vp-1|-K\ge AM-K\ge M$.
Part (v) follows similarly, since $|\vp|\ge 1+\ap_0/M$
and therefore $|S_4|\ge K(|\vp|-1)^2 \ge K\ap_0^2/M^2$.
\end{proof}

\subsection{Bounds for roots relevant to instability}

Next we focus on roots of $F_0$ that may be related to eigenvalues
$\lambda$ of the matrix $B$ having non-negative real part.
It turns out these are roots $\vp$ in the punctured annulus $D_a$ of 
Proposition~\ref{p:F0roots} that are near $1$.  
Recall the relation
\eqref{dd:lambdavp}
between eigenvalues of the matrix $B$ and roots $\vp$ of $F$, namely 
\[
\lambda=(A-\vp\inv)(\vp-1).
\] 

\begin{lemma}\label{l:phibound}
Under the conditions of Proposition~\ref{p:F0roots},
if $M$ is large enough, then whenever \eqref{dd:lambdavp} 
holds with $\vp\in D_a$, then $\re\lambda\ge 0$ implies
\begin{equation}\label{c:phibound}
1\le \re\vp<1+\frac{2\ap_0}M \quad\mbox{ and }\quad 
|\im\vp|< \frac{2\ap_0}{M^{3/4}}\,.
\end{equation}
\end{lemma}
\begin{proof}
By \eqref{dd:lambdavp}, $\lambda = A\vp-A-1+\vp\inv$, hence
if $\re\vp<0$ then $\re\lambda<-A-1$.
Writing
\[
\qquad \mu=\re\vp-1,\quad \nu=\im\vp,  
\]
we then have $\mu\ge-1$ and
\begin{equation}\label{e:mu1}
0\le \re\lambda =  A\mu-1 + (1+\mu)|\vp|^{-2} =
(A+|\vp|^{-2})\mu - 1+|\vp|^{-2}\,.
\end{equation}
For $\vp\in D_a$ and $M$ large, we infer
$|\vp|^{-2}\le M^{2\am/M}\le 1+ {4\am}M\inv \log M$, 
then
\begin{equation}\label{e:mu2}
\mu\ge -\frac{4\am \log M}M > -\frac KA \,.
\end{equation}
Now because $|\vp|^{-2}\le(1+\mu)^{-2}$, we deduce from \eqref{e:mu1} that 
\[
0\le (A\mu-1)(\mu+1) +1  = (A\mu+K)\mu \,.
\]
This entails $\mu\ge0$, due to \eqref{e:mu2}. 
Since $|\vp|\inv\ge 1-\frac{\ap_0}M$ implies $\re\vp<1+\frac{2\ap_0}M$,
we have established the desired bounds on $\re\vp$.

Now since $|\vp|^2=(1+\mu)^2+\nu^2$ and $0\le\mu<\frac{2\ap_0}M$, 
we deduce from \eqref{e:mu1} that 
\[
\nu^2 \le 
\frac{1+\mu}{1-A\mu} - (1+\mu)^2 =
\frac{1+\mu}{1-A\mu}(K+A\mu)\mu
<2K\mu<\frac{4\beta \ap_0}{M^{3/2}} \,.
\]
Since we may presume $\beta\le \ap_0$, therefore $|\nu|< 2\ap_0 M^{-3/4}$ as claimed. 
\end{proof}

Any roots of $F_0$ in the region where \eqref{c:phibound} holds
actually satisfy a tighter bound, namely $|\vp-1| = O(1/M)$,
as we now show. 

\begin{proposition}\label{p:F0roots1}
Under the conditions of Proposition~\ref{p:F0roots}, 
there exist positive constants $\ap_1=\ap_1(\beta)$ and $M_1=M_1(\beta)$
such that whenever $M>M_1$, 
any zeros $\vp\in D_a$ of $F_0$ that satisfy the bounds in 
\eqref{c:phibound} must satisfy $|\vp-1|\le \frac {\ap_1}M$.
Moreover, 
\begin{equation}\label{e:F0near1}
|F_0(\vp)| \ge  \frac {K{\ap_1}^2}{4M} \,,
\end{equation}
for all $\vp$ that satisfy 
\begin{equation}
0 \le \re(\vp-1) \le \frac {\ap_1}M \quad\mbox{and}\quad
\frac {\ap_1}M \le |\vp-1| \le  \frac{3\ap_0}{M^{3/4}}\,.
\label{b:bbounds}
\end{equation}
\end{proposition}
\begin{proof} In the expression $F_0=\vp^MP_2-P_1$ we seek to show 
that the first term dominates, provided \eqref{b:bbounds} holds for some $\ap_1$.
Writing $\zeta = A(\vp-1)$ for convenience, we have $A\vp-1=K+\zeta$,
so by \eqref{e:P2-1},
\[
P_2 = K^3(1+\frac \zeta K)^3\vp 
+ \zeta^2 KM\left(1+\frac \zeta K-\frac 1{KM}\right) \,.
\]
By \eqref{b:bbounds} we have $\zeta/K=O(M^{-\frac14})$
and $|\vp|\le 1+O(M^{-3/4})$, so 
\[
|P_2| \ge  KM|\zeta|^2 (1-O(M^{-\frac14})) - K^3(1+O(M^{-\frac14})).
\]
Because 
\[
\frac{K^3}{KM|\zeta|^2} \le \frac{\beta^2}{M^2|\zeta|^2}
\le \frac{\beta^2}{\ap_1^2}, 
\]
for $\ap_1>2\beta$ and large enough $M$ we infer that
\begin{equation}
|P_2| \ge \frac12 KM|\zeta|^2\,.
\label{bb:P2}
\end{equation}
On the other hand, due to \eqref{b:Amhalf} we have
\[
|S(\vp)|\le |\zeta||K+\zeta|^2 2|\vp| = K^2|\zeta|(1+O(M^{-\frac14})) \,,
\]
therefore from \eqref{e:P1-1} we obtain the upper bound
\begin{align}
|P_1|&\le (K^3 + K|\zeta|^2 + K^3M|\zeta|)
(1+O(M^{-\frac14}))
\nonumber\\
&\le 2KM|\zeta|^2\left( \frac{\beta^2}{\ap_1^2} + \frac1M 
+ \frac{\beta^2}\ap_1 \right)
\nonumber\\
&\le \frac18 KM|\zeta|^2\,
\label{bb:P1}
\end{align}
if $\ap_1>40\beta^2$, say, and $M$ is large enough.
Since $|\vp|\ge1$ if \eqref{b:bbounds} holds, the result follows.
\end{proof}

\subsection{Convergence of \texorpdfstring{$K^{-3}F_0(1+z/M)$}{K-3 F0(1+z/M)}}
After the results of the previous subsection, 
to study unstable eigenvalues of $B$
we are motivated to make the change of variables 
\[
\vp = 1+\frac zM
\]
as in Section~\ref{s:formal}. According to Proposition~\ref{p:F0roots1},
for any zeros $\vp\in D_a$ of $F_0$ that correspond to $\re\lambda\ge0$,
the quantity $z=M(\vp-1)$ must satisfy
\begin{equation}
\re z\ge0, \qquad 0<|z|\le \ap_1.
\label{b:z0}
\end{equation}
As in \eqref{d:kappa},
let us now {\em define} $\kk=\kk(K,M)= K\sqrt{M}$ and $\eps=1/\sqrt{M}$.
Then the formal approximations in Section~\ref{s:formal} are
rigorous, with errors that are uniform over the values of 
$(z,\kk)\in \CC\times\CC$
such that 
\begin{equation}
|z|\le \hat\ap,
\label{b:z1}
\end{equation}
where $\hat\ap>\ap_1$ is an arbitrary constant (to be chosen later), and 
\begin{equation}
\frac1{2\beta}\le |\kk|\le 2\beta, \quad |\arg \kk|< 2\hat \gamma \,,
\label{b:alpha1}
\end{equation}
for some small $\hat\gamma>0$.
(We allow $\kk$ to be complex with small argument
here, to simplify derivative estimates later.)
By consequence, the convergence in \eqref{d:Phiza} holds,
in the following sense. 
\begin{proposition}\label{p:F0converge}
Uniformly for $(z,\kk)$ satisfying \eqref{b:z1}--\eqref{b:alpha1},
with $M=1/\eps^2$ and $K=\kk{\eps}$ we have that 
\begin{equation}
K^{-3}F_0\left(1+\frac zM\right) \to Q(z;\kk)
\qquad \mbox{as $\eps\to0$}.
\end{equation}
\end{proposition}

\section{Analysis of roots of \texorpdfstring{$F$}{F}}

Recall from \eqref{F:sort} we have
\[
F(\vp)= F_0(\vp) +A^{-M}F_1(\vp)\,,
\qquad F_1(\vp)= R_1+\vp^{-M}R_2
\]
where $R_1$, $R_2$ are low-degree polynomials 
that may be written in the form
\begin{align} 
\label{e:r1}
R_1 &= -A^2( (A\vp-1)^2\vp^2-(\vp-1)^2),
\\
\label{e:r2}
R_2 &= A(\vp-1)^3 + A^2\vp(A\vp-1)^2(M(\vp-1)+\vp).
\end{align}
For large $M$, $A^{-M}$ is exponentially small, with the bound
\[
A^{-M}=(1+K)^{-M} \le (1+\frac1{\beta\sqrt{M}})^{-M}\le e^{-\sqrt{M}/2\beta}\,.
\]

We now roughly characterize the location of the $2M+4$ roots of $F$.

\begin{proposition}\label{p:Froots}
Under the conditions of Proposition~\ref{p:F0roots}, 
there exists $M_2=M_2(\beta)$ such that whenever $M>M_2$, 
$F$ has (counting multiplicities):
\begin{itemize}
\item[(i)] one double root at $\vp=1$, and one double root at $\vp=A\inv$.
\item[(ii)] $M-2$ roots in the punctured annulus $D_a$, and
$M-2$ roots with $(A\vp)\inv\in D_a$ which 
satisfy $|\vp| <1-\frac34 K<M^{-\am/M}$.
\item[(iii)] one simple real root in $B(-M,1)$, and one 
with $(A\vp)\inv\in B(-M,1)$.
\item[(iv)] one simple real root at $\vp= A^{-1/2}$ and one at $\vp= -A^{-1/2}$.
\end{itemize}
\end{proposition}

\begin{proof}
We note that due the root symmetry \eqref{root:sym}, 
the multiplicity of each root $\vp$ of $F$ 
is the same as the multiplicity of $1/(A\vp)$, 
unless $\vp =\pm A^{-1/2}$. 
Also, all non-real roots of $F$ come in complex-conjugate pairs
when $K$ is real.

For $|\vp|\ge M^{-\am/M}$ we then have $|\vp|^{-M}<M^\am$
and it follows $A^{-M}|F_1(\vp)|$ is exponentially small.

Combining the lower bounds in parts (iii)--(v) of 
Corollary~\ref{c:F0ests} with the count of roots of $F_0$ 
in parts (i), (iii) and (iv) of Proposition~\ref{p:F0roots},
we conclude from Rouch\'e's theorem that $F$ has 
a simple root inside the ball $B(-M,1)$, 
and $M$ roots inside the closed annulus $D_a\cup\{1\}$,
the same as $F_0$. 

By examining \eqref{e:r1}--\eqref{e:r2}, we find 
$\vp^M F_1= \vp^M R_1+R_2$ has at least a double root at $\vp=1$,
due to the fact that the expression
\[
-\vp^M\vp + M(\vp-1)+\vp = (\vp-1)\left(M-\sum_{j=1}^M\vp^j\right) 
\]
has a double root at $\vp=1$. Then, because $A^{-M}F_1''(1)$
is exponentially small, it follows from \eqref{e:F0pp}
that $F''(1)\ne0$. This proves (i). 

Now (iii) follows and also  (ii), due to the fact that for
$\vp\in D_a$ and $M$ large,
\[
|A\vp|\inv < \frac{\exp({\am}M\inv\log M)}{1+K} < 1-\frac{3K}4
< \exp(-\am M\inv \log M).
\]
To infer (iv) we can simply recall that we know $F(\pm A^{-1/2})=0$
due to the root symmetry relation \eqref{root:sym}.
These roots must be simple, since we have accounted for all
$2M+4$ roots of $F$.
\end{proof}

Next, we can characterize zeros of $F$ that may correspond to unstable 
eigenvalues of $B$ as follows.

\begin{proposition}\label{p:Froots1}
Under the conditions of Propositions~\ref{p:Froots} and \ref{p:F0roots1},
there exists $M_3=M_3(\beta)$ such that whenever $M>M_3$ and 
$\lambda$ is an eigenvalue of $B$ with $\re\lambda\ge0$, 
then $\lambda=(A-\vp\inv)(\vp-1)$
for some root $\vp$ of $F$ that satisfies
\begin{equation}\label{b:vp0}
\re\vp \ge1\,, \qquad |\vp-1|\le \frac{\ap_1} M \,.
\end{equation}
\end{proposition}

\begin{proof} Under the correspondence between $\lambda$ and $\vp$
in \eqref{dd:lambdavp}, the zeros of $F$ described in parts (iii) and (iv) of 
Proposition~\ref{p:Froots} correspond to negative real values of $\lambda$,
and the roots in part (i) correspond to $\lambda=0$. 
So, given $M$ is large enough,
for any nonzero eigenvalue $\lambda$ satisfying $\re\lambda\ge0$,
necessarily  \eqref{dd:lambdavp} holds for some $\vp\in D_a$. 
This $\vp$ must satisfy the bounds in \eqref{c:phibound}, 
due to Lemma~\ref{l:phibound}.
For these values of $\vp$, we have $|F_1(\vp)|\le C M^{1/4}$, 
so $A^{-M}|F_1|$ is exponentially small.
Then we can conclude from Proposition~\ref{p:F0roots1} that 
\begin{equation}\label{e:Flower}
|F(\vp)| \ge \frac{K\ap_1^2}{8M} >0\,,
\end{equation}
for all $\vp$ that satisfy \eqref{b:bbounds}. The conclusion follows.
\end{proof}

Further, 
the convergence in Proposition~\ref{p:F0converge} holds with
$F$ in place of $F_0$:
\begin{proposition}\label{p:Fconverge}
Let $\hat\ap>\ap_1$, and let $\hat\gamma>0$ be small.
Uniformly for $(z,\kk)$ satisfying \eqref{b:z1}--\eqref{b:alpha1},
with $M=1/\eps^2$ and $K=\kk{\eps}$ we have that 
\begin{equation}
Q^\eps(z;\kk):=K^{-3}F\left(1+\frac zM\right) \to Q(z;\kk)
\qquad \mbox{as $\eps\to0$}.
\end{equation}
Furthermore, for each pair of integers $j,k\ge0$, the derivatives
\begin{equation}\label{e:dconv}
\D_z^j\D_\kk^k Q^\eps(z;\kk) \to 
\D_z^j\D_\kk^k Q(z;\kk) \quad\mbox{as $\eps\to0$},
\end{equation}
uniformly for all $z$ and $\kk$ satisfying 
\begin{equation}
|z|\le  \hat\ap,  \qquad  
\beta\inv\le |\kk|\le \beta, \quad |\arg \kk|< \hat\gamma \,.
\label{b:z2}
\end{equation}
\end{proposition}
\begin{proof}
For $|z|\le\hat\ap$ and $\vp=1+\frac zM$, the factor $|\vp|^{-M}$
is bounded by $e^{2\hat\ap}$. Hence again $A^{-M}F_1$
is exponentially small, and the convergence of $Q^\eps(z;\kk)$
follows from Proposition~\ref{p:F0converge}.

The convergence of derivatives follows from the Cauchy integral formula
representation for such derivatives,
since $Q^\eps(z;\kk)$ is analytic for $z$ satisfying \eqref{b:z1}
and $\kk$ satisfying \eqref{b:alpha1}.
\end{proof}

{\em Curves of roots.}
Recall that the non-real roots $z$ of $Q=Q(z;\kappa)$ are simple 
and those that may satisfy $\re z\ge0$
lay on the curves $z_j^0(\kappa)$ described by Lemma~\ref{lem:Q0zj}.
Moreover, due to \eqref{d:zk} and \eqref{e:dzk},
only a finite number of these curves provide values 
that can satisfy \eqref{b:z0}, corresponding to values
of $\vp=1+\frac zM$ that satisfy \eqref{b:vp0}.
In particular, we note the following. 
\begin{corollary} For $j\in\N$, if $\beta\ge\kk^0_j$ and
$\ap_1$ is given by Proposition~\ref{p:F0roots1},
then 
\[
\ap_1> |z_j^0(\kk)| \quad\mbox{for all $\kk\in[\kk^0_j,\beta]$.}
\]
\end{corollary}
\begin{proof} 
Suppose $\ap_1\le |z_j^0(\kk)|$ for some 
$\kk\in[\kk^0_j,\beta]$.
Recall $z=z_j^0(\kk)$ satisfies $Q(z,\kk)=0$, $\re z\ge0$.
Then for $M$ large enough, 
$\vp=1+\frac zM$ satisfies \eqref{b:bbounds}, and
\begin{equation}\label{b:Qepslower}
|Q^\eps(z;\kk)| \ge \frac{\ap_1^2}{8K^2M}\ge \frac{\ap_1^2}{8\beta^2}>0\,,
\end{equation}
due to \eqref{e:Flower}. 
But this contradicts the convergence result in Proposition~\ref{p:Fconverge}.
\end{proof}

Any finite number of the curves $z_j^0$ of simple zeros of $Q$ 
perturb to curves $z_j^\eps$ of simple zeros of $Q^\eps$ 
as a consequence of the implicit function theorem, as follows.

\begin{proposition}\label{p:zkeps}
For $j\in\N$, suppose $\beta>\kk^0_j$.  
Let $\ap_1$ be given by Proposition~\ref{p:F0roots1},
and suppose
\[
\hat\ap > 
|z_j^0(\kk)| \quad\mbox{for all $\kk\in[\beta\inv,\beta]$.}
\]
Then for sufficiently small $\eps>0$, there is a curve 
$z_j^\eps:[\beta\inv,\beta]\to B(0,\hat\ap)$
that is real analytic, with the following properties: 

\begin{itemize}
\item[(i)]
For each $\kappa\in[\beta\inv,\beta]$, $z_j^\eps(\kk)$ is a simple
root of $Q^\eps(z,\kk)$. 
\item[(ii)]
$z_j^\eps(\kk)\to z_j^0(\kk)$ as $\eps\to0$, 
uniformly for $\kk\in[\beta\inv,\beta]$,
together with any finite number of derivatives in $\kk$.
\item[(iii)] There exists $\zeta_j<\beta$
satisfying $\zeta_j^\eps\to\kk_j^0$ as $\eps\to0$,
such that $\re z_j^\eps(\kk)\ge0$ if and only if 
$\kk\ge \zeta_j^\eps$, and 
\begin{equation} \label{e:dzetak}
\re \frac{dz_j^\eps}{d\kk} >0 \quad\mbox{and}\quad
\im \frac{dz_j^\eps}{d\kk} >0 
\quad\mbox{for all $\kk\in[\zeta_j^\eps,\beta]$.}
\end{equation}
\end{itemize}
\end{proposition}
\begin{proof}
The existence of the curve, its analyticity in $\kk$,
and properties (i), (ii) and (iii),
follow from standard implicit function theorem arguments
using the simplicity of the roots of $Q$, the convergence in 
Proposition~\ref{p:Fconverge}, and Lemma~\ref{lem:Q0zj}.
\end{proof}

\section{Analysis of eigenvalues of \texorpdfstring{$B$}{B}}\label{s:eigsB}

The $M$ eigenvalues $\lambda$ of $B$ are generated via the relation
\eqref{dd:lambdavp} by: 
one of the roots of $F$ at $\vp=1$, the one near $-M$, and the 
$M-2$ roots in $D_a$.
The roots $\pm A^{-1/2}$, one root at $1$, and one root at $A\inv$
are spurious, as discussed earlier.
We have not characterized the multiplicity of all the 
eigenvalues or all the roots,
but each eigenvalue must correspond to some root of $F$,
and vice versa.

\subsection{Curves of unstable eigenvalues}
Recall that zeros $z$ of $Q^\eps(z;\kk)$ correspond to eigenvalues $\lambda$ of 
the matrix $B$ via the relation \eqref{e:lambdaz}.
We rescale this relation by defining 
\begin{equation}\label{d:Lamz}
\Lambda(z;\kk,\eps) = \frac{M\lambda}{K} = z + \frac{\eps}{\kk}\frac{z^2}{1+\eps^2z}.
\end{equation}
Clearly $\Lambda(z;\kk,\eps)\to z$ as $\eps\to0$, together with derivatives,
uniformly for $z$, $\kk$ satisfying \eqref{b:z2}.

When $\eps=0$, of course we have $\re\Lambda(z;\kk,0)\ge0$ 
if and only if $x=\re z\ge0$,
for any $\kk>0$. By stardard implicit function theorem arguments,
for small enough $\eps>0$ there is a real analytic function 
$(y,\kk)\mapsto \hat x(y,\kk,\eps)$ such
that  for $|z|\le\hat\alpha$ and $\kappa\in[\beta\inv,\beta]$,
\[
\re \Lambda(x+iy;\kk,\eps)\ge 0 \quad\mbox{if and only if} 
\quad x\ge \hat x(y,\kk,\eps) \,.
\]
Let $\calI_\eps\subset B(0,\hat\alpha)\times [\beta\inv,\beta]$ denote the surface on which this holds,
i.e., where $\re\Lambda=0$. When $\eps=0$, the imaginary axis $\calI_0$ 
meets each curve $z_j^0$ transversely 
due to the computation in \eqref{e:redzdk}. 
Therefore, for sufficiently small $\eps>0$,
the surface $\calI_\eps$ meets each curve $z_j^\eps$ 
provided by Proposition~\ref{p:zkeps} transversely. 
By consequence, each curve given by 
\begin{equation}\label{d:lambdak}
\lambda_j^\eps(\kk) = \kk\eps^3\,\Lambda(z_j^\eps(\kk),\kk,\eps) \,,
\qquad \kk\in[\beta\inv,\beta],
\end{equation}
provides a curve of eigenvalues of $B$ that
must cross the imaginary axis transversely as $\kappa$ increases,
exactly once for $\kk\in[\beta\inv,\beta]$.

\subsection{Proof of Theorem~\ref{thm:eigsB}}
Let $\beta_0\in(0,1)$ and $k\in\N$. Recalling that the curves $z_j^0(\kappa)$ 
and numbers $\kappa^0_j>1$ were defined in \eqref{d:zk}, we fix 
$\beta_k \in (\kappa^0_k,\kappa^0_{k+1})$, 
and note
\[
\re z_j^0(\beta_0)<0 \quad\mbox{for all $j$}, \qquad
\re z_j^0(\beta_k)
\begin{cases}>0 &\quad\mbox{for all $j\le k$},\\
<0 &\quad\mbox{for all $j>k$.}
\end{cases}
\]
Next, choose $\beta>\max(\beta_k,\beta_0\inv)$, let $\alpha_1=\alpha_1(\beta)$
be determined by Proposition~\ref{p:F0roots1},
and choose $\hat\ap>\ap_1$
such that 
\[
|z_j^0(\kappa)|\le \hat \ap \qquad\mbox{for all $\kappa\in[\beta\inv,\beta]$,
\ \ $j=1,\ldots,k$.}
\]
If $M$ is sufficiently large (i.e., $M>M_{0,k}$ for some
$M_{0,k}$ depending on $\beta$) 
then analytic curves 
$z_j^\eps(\kk)$ are defined by Proposition~\ref{p:zkeps} 
and $\lambda_j^\eps(\kk)$ by \eqref{d:lambdak}. Let
\begin{equation}
\lambda_j(\kappa)=\lambda^\eps_j(\kappa),
\quad \kappa\in[\beta\inv,\beta],
\quad j=1,\ldots,k.
\end{equation}
Due to Propositions~\ref{p:Froots1} and \ref{p:zkeps} and the discussion
above, each curve $\lambda_j$ crosses the imaginary axis transversely
at some point $\kappa_j=\kappa_j^\eps\in[\zeta_j^\eps,\beta]$ 
that satisfies 
\begin{equation}\label{e:kappalim}
\kappa_j^\eps \to \kappa_j^0 \quad\mbox{as $\eps\to0$}.
\end{equation}
By consequence, 
for small enough $\eps>0$ we have 
$\kk_{j-1}^\eps<\kk_j^\eps<\beta_k$ for $j=1,\ldots,k$,
where we set $\kk_0^\eps=\beta_0$. Also we have the monotonicity
relations in \eqref{e:dlamdk}.

Since $|z_j^\eps(\kk)|\le\hat\ap$, 
the eigenvalues of $B$ given by $\lambda_j(\kk)$, $j=1,\ldots,k$ 
satisfy the bound
\begin{equation}
|\lambda_j(\kk)|\le 2\kappa\eps^3\hat\ap \le \hat C_k M^{-3/2}
\end{equation}
for $M$ large.  Furthermore, due to Lemma~\ref{lem:simpleBF} 
(proved below), every such eigenvalue 
$\lambda_j(\kk)$ is a simple eigenvalue of $B$,
since the roots $z=z_j^\eps(\kk)$ of $Q^\eps(z;\kk)$ are simple.

It remains to prove that for $\kk\in[\beta_0,\beta_k]$,
if $\hat\lambda\ne0$ is an eigenvalue of $B$ with $\re\hat\lambda\ge0$,
and $\im\lambda\ge0$, 
then necessarily $\hat\lambda=\lambda_j(\kk)$ for some $j\le k$
with $\kk\ge\kk_j^\eps$.  
According to Proposition~\ref{p:Froots1}, necessarily 
such an eigenvalue must satisfy 
\[
\hat\lambda = \kk\eps^3 \Lambda(\hat z;\kk,\eps)\,,
\]
where $Q^\eps(\hat z;\kk)=0$, $\re \hat z\ge0$ and $|\hat z|\le\alpha_1$.

Now, for any $r>0$ sufficiently small, note that the balls 
$B(z_j^0(\kk),r)$ do not overlap or contain $0$ for any $\kk$, 
and each must contain a simple root $z_j^\eps(\kk)$ of $Q^\eps(z,\kk)$.
Fix some such $r>0$, and let $\Omega_r$ be the set of $(z,\kk)$ such that 
\[
\re z\ge0,\quad \im z\ge0, \quad 0<|z|\le \hat\ap,\quad |z-z_j^0(\kk)|\ge r
\ \ \mbox{for $j=1,\ldots,k$},
\]
and $\kk\in[\beta_0,\beta_k]$.
Because $\beta_k<\kk^0_{k+1}$, for sufficiently small $r>0$
we have
\[
\hat \mu(r):= \inf_{\Omega_r} |Q(z,\kk)/z^2| >0\,.
\]
From the convergence in Proposition~\ref{p:Fconverge} it follows
\[
\hat \mu^\eps(r):= \inf_{\Omega_r} |Q^\eps(z,\kk)/z^2| >0\,,
\]
if $\eps>0$ is sufficiently small. Then it follows that 
$|\hat z-z_j^0|<r$ for some $j\le k$, whence necessarily
$\hat z = z_j^\eps(\kk)$. And $\kk\ge\kk^\eps_j$ since $\re\hat z\ge0$.

This completes the proof of Theorem~\ref{thm:eigsB}.

\subsection{Simplicity of eigenvalues}
It remains to prove Lemma~\ref{lem:simpleBF}, which shows 
in particular that simple roots of $F$ provide simple eigenvalues
of $B$.
\begin{proof}[Proof of Lemma~\ref{lem:simpleBF}]
First, we show that the kernel of $B-\lambda I$ is one-dimensional.
Recall from Section~\ref{s:reform} 
that whenever $(B-\lambda I)V=0$, then the components $V_\ell$
have the form \eqref{A1E4} for some constants $c_1$, $c_2$. 
More generally, if $V=V(\vp)$ has the form \eqref{A1E4}
with $\vp_1=\vp$, $\vp_2=(A\vp)\inv$,
and if $\lambda(\vp)=(A-\vp\inv)(\vp-1)$, then
equations \eqref{S1E6}--\eqref{S1E7} are equivalent to the equation
\begin{equation}\label{e:S9e1}
(B-\lambda(\vp) I) V(\vp)
= [e_m, e_1]\calD(\vp) 
\begin{pmatrix}c_1\\c_2\end{pmatrix}  = 0\,,
\end{equation}
where $e_j$ denotes the $j$th standard basis vector, and
\begin{equation}
\calD(\vp)=
\begin{pmatrix}
f\left(  \varphi_{1}\right)  & f\left(  \varphi_{2}\right) \\
g\left(  \varphi_{1}\right)  & g\left(  \varphi_{2}\right)
\end{pmatrix} .
\end{equation}
The value $\lambda$ is an eigenvalue if and only if $\calD(\vp)$
is singular. 
The matrix $\calD(\vp)$ does not vanish in this case, however,
for the following reason. 
Since $S(\vp)\ne0$ and $A\vp_1\vp_2=1$, necessarily $\vp_1$ and $\vp_2$
are distinct and have the same sign. 
But the function $\vp f(\vp)=A\vp^M+\vp-1$ is convex and cannot have two
distinct roots with the same sign.   
Hence it is not possible that $f(\vp_j)=0$ for both $j=1$ and $2$.

It follows that the kernel of $B-\lambda I$ is one dimensional,
and the eigenspace is spanned by $V(\vp)$, taking 
\[
\begin{pmatrix}
c_1\\c_2
\end{pmatrix}=
\begin{pmatrix}
f(\vp_2)\\-f(\vp_1)
\end{pmatrix}.
\]
Next, we determine when $\lambda$ is simple, 
i.e., when it has algebraic multiplicity one.
Since $(B-\lambda I)V=0$, this is the case if and only if the equation 
\begin{equation}\label{e:BU}
(B-\lambda I)U=V
\end{equation}
has no solution. 
Letting $'$ denote differentiation with respect to $\vp$, 
it follows by differentiating \eqref{e:S9e1}
(while keeping $c_1$, $c_2$ fixed), that 
\[
(B-\lambda I)V' = \lambda' V + 
[e_m,e_1]\calD'(\vp)
\begin{pmatrix}c_1\\c_2\end{pmatrix} \,.
\]
Now, $\lambda'= A-\vp^{-2}\ne0$ whenever $\vp\ne\pm A^{-1/2}$,
so it follows that a solution to \eqref{e:BU}
exists if and only if $\lambda' U = V'(\vp)-\hat U$
where $\hat U$ is a solution to 
\[
(B-\lambda I)\hat U = 
[e_m,e_1]\calD'(\vp)
\begin{pmatrix}c_1\\c_2\end{pmatrix} \,.
\]
As in Section~\ref{s:reform}, necessarily
$\hat U_\ell = \hat c_1 \vp_1^{M-\ell}+ \hat c_2 \vp_2^{M-\ell}$
for some constants $\hat c_1$, $\hat c_2$ that satisfy
\begin{equation}\label{e:hatc}
\cal D(\vp)
\begin{pmatrix}\hat c_1\\\hat c_2\end{pmatrix} 
=
 \calD'(\vp)
\begin{pmatrix}c_1\\c_2\end{pmatrix} \,.
\end{equation}
Writing $f_j=f(\vp_j)$, $f_j' = f'(\vp_j)\vp_j'$ and 
similarly for $g_j$, $g_j'$, the fact that $\calD(\vp)$ is singular means
\begin{equation}\label{e:del00}
\delta(\vp)= f_1g_2-g_1f_2 = 0,
\end{equation}
and a left null vector is given by $(g_1,-f_1)$ or 
$(g_2,-f_2)$ (since $\calD(\vp)\ne0$).
Supposing $f_1\ne0$, applying the left null vector 
to \eqref{e:hatc}
we find that a solution of \eqref{e:hatc} exists if and only if
\begin{align*}
0 = (g_1,-f_1) 
\begin{pmatrix}
f_1' & f_2'\\ g_1' & g_2'
\end{pmatrix}
\begin{pmatrix}f_2\\ -f_1 \end{pmatrix}
&=
g_1(f_1'f_2-f_2'f_1)+f_1(f_1 g_2' -f_2 g_1')
\\&=
f_1 \delta'(\vp),
\end{align*}
where we used \eqref{e:del00} to replace $g_1f_2$ by $f_1g_2$.
If $f_2\ne0$ similarly the criterion is $0=f_2\delta'(\vp)$.
Thus an eigenvalue $\lambda$ is simple if and only if $\delta'(\vp)\ne0$,
and this is equivalent to $F'(\vp)\ne0$.
\end{proof}

\section*{Acknowledgements}
The authors acknowledge support
from the Hausdorff Center for Mathematics and the CRC 1060 on {\it Mathematics of emergent effects}, Universit\"at Bonn. 
This material is based upon work supported by the National
Science Foundation under grants 
DMS 1515400 and 1812609,
partially supported by the Simons Foundation under grant 395796, 
and by the NSF Research Network Grant no.\ RNMS11-07444 (KI-Net).

\bibliographystyle{siam}
\bibliography{bib-hopf.bib}

\end{document}